\documentclass[11pt]{amsart}

\usepackage[
paper=a4paper,
text={147mm,230mm},centering
]{geometry}

\iftrue
\makeatletter
\def\@settitle{%
  \vspace*{-20pt}
  \begin{flushleft}%
    \baselineskip14\p@\relax
    \normalfont\bfseries\LARGE
    \@title
  \end{flushleft}%
}
\def\@setauthors{%
  \begingroup
  \def\thanks{\protect\thanks@warning}%
  \trivlist
  \large \@topsep30\p@\relax
  \advance\@topsep by -\baselineskip
  \item\relax
  \author@andify\authors
  \def\\{\protect\linebreak}%
  \authors
  \ifx\@empty\contribs
  \else
    ,\penalty-3 \space \@setcontribs
    \@closetoccontribs
  \fi
  \normalfont
  \endtrivlist
  \endgroup
}
\def\@setabstracta{%
    \ifvoid\abstractbox
  \else
    \skip@25\p@ \advance\skip@-\lastskip
    \advance\skip@-\baselineskip \vskip\skip@
    \box\abstractbox
    \prevdepth\z@ 
    \vskip-10pt
  \fi
}
\renewenvironment{abstract}{%
  \ifx\maketitle\relax
    \ClassWarning{\@classname}{Abstract should precede
      \protect\maketitle\space in AMS document classes; reported}%
  \fi
  \global\setbox\abstractbox=\vtop \bgroup
    \normalfont\small
    \list{}{\labelwidth\z@
      \leftmargin0pc \rightmargin\leftmargin
      \listparindent\normalparindent \itemindent\z@
      \parsep\z@ \@plus\p@
      
    }%
    \item[\hskip\labelsep\bfseries\abstractname.]%
}{%
  \endlist\egroup
  \ifx\@setabstract\relax \@setabstracta \fi
}
\def\section{\@startsection{section}{1}%
  \z@{-1.2\linespacing\@plus-.5\linespacing}{.8\linespacing}%
  {\normalfont\bfseries\large}}
\def\subsection{\@startsection{subsection}{2}%
  \z@{-.8\linespacing\@plus-.3\linespacing}{.3\linespacing\@plus.2\linespacing}%
  {\normalfont\bfseries}}
\def\subsubsection{\@startsection{subsubsection}{3}%
  \z@{.7\linespacing\@plus.1\linespacing}{-1.5ex}%
  {\normalfont\itshape}}
\def\@secnumfont{\bfseries}
\makeatother
\fi 

\usepackage{amssymb}
\usepackage{mathrsfs}
\usepackage[all]{xy}
\usepackage{graphicx,color,float}
\usepackage[bookmarks]{hyperref}
\usepackage{pinlabel}
\usepackage[textwidth=1.3in,color=yellow]{todonotes}
\usepackage{amsmath}
\usepackage{pb-diagram,pb-xy}
\usepackage{amsmath}
\usepackage{amsfonts}
\usepackage{graphicx}
\usepackage{amscd}
\usepackage{mathtools}
\usepackage{url}
\usepackage{caption}
\usepackage{subcaption}
\usepackage{fancybox}
\usepackage{wrapfig}
\usepackage{color}
\usepackage{multirow, url}
\usepackage{tikz}
\usetikzlibrary{shapes.geometric, arrows}
\usetikzlibrary{shapes}
\usepackage{xcolor}
\usepackage[normalem]{ulem}  

\tikzstyle{startstop} = [rectangle, rounded corners, 
minimum width=3cm, 
minimum height=1cm,
text centered, 
draw=black]

\tikzstyle{io} = [trapezium, 
trapezium stretches=true, 
trapezium left angle=70, 
trapezium right angle=110, 
minimum width=3cm, 
minimum height=1cm, text centered, 
draw=black, fill=blue!30]

\tikzstyle{process} = [rectangle, rounded corners, 
minimum width=4cm, 
minimum height=1cm,
text centered, 
draw=black]

\tikzstyle{decision} = [diamond, 
minimum width=3cm, 
minimum height=1cm, 
text centered, 
draw=black, 
fill=green!30]
\tikzstyle{arrow} = [thick,->,>=stealth]

\textwidth=\paperwidth \advance\textwidth by-2\oddsidemargin
\advance\oddsidemargin by-1in
\evensidemargin=\oddsidemargin
\calclayout

\theoremstyle{plain}
\newtheorem{theorem}{Theorem}[section]

\newtheorem{proposition}[theorem]{Proposition}
\newtheorem{lemma}[theorem]{Lemma}

\newtheorem{corollary}[theorem]{Corollary}

\newtheorem{setup}[theorem]{Setup}

\theoremstyle{definition}

\newtheorem{definition}[theorem]{Definition}
\newtheorem{example}[theorem]{Example}

\theoremstyle{remark}
\newtheorem{remark}[theorem]{Remark}

\newcommand{\C}{\mathbb{C}}

\newcommand{\R}{\mathbb{R}}

\newcommand{\Z}{\mathbb{Z}}

\newcommand{\CP}{\mathbb{C}P}

\newcommand{\git}{\mathbin{/\mkern-6mu/}}

\def\mcal{\mathcal}

\numberwithin{equation}{section} \numberwithin{table}{section}

\def\to{\mathchoice{\longrightarrow}{\rightarrow}{\rightarrow}{\rightarrow}}
\makeatletter
\newcommand{\shortxra}[2][]{\ext@arrow 0359\rightarrowfill@{#1}{#2}}
\def\longrightarrowfill@{\arrowfill@\relbar\relbar\longrightarrow}
\newcommand{\longxra}[2][]{\ext@arrow 0359\longrightarrowfill@{#1}{#2}}

\makeatother
\numberwithin{equation}{section}

\begin{document}                                                                          
\title[From flag varieties to basic affine spaces]{Lifting holomorphic disks from flag varieties to\\
basic affine spaces}

\author{Yoosik Kim}
\address{Department of Mathematics and Institute of Mathematical Science, Pusan National University, Busan, Republic of Korea}
\email{yoosik@pusan.ac.kr}
\thanks{The research of Y.  \ Kim was supported by the National Research Foundation of Korea(NRF) grant funded by the Korea government (MSIT) (RS-2025-16069532 and RS-2020-NR049535).}


\begin{abstract}
Let $G$ be a complex reductive algebraic group, the complexification of a compact Lie group $K$. Consider a holomorphic principal $G$-bundle whose total space contains a $K$-invariant Lagrangian submanifold $L$. We develop a method for lifting holomorphic disks from the base with boundary on $L/K$ to the principal $G$-bundle. We show that the Gross--Hacking--Keel--Kontsevich superpotential restricted to a distinguished class of seeds of the basic affine space is obtained by lifting holomorphic disks in the corresponding flag manifold.
\end{abstract}

\maketitle
\setcounter{tocdepth}{1} 
\tableofcontents

\section{Introduction}

In \cite{GHKK18}, Gross--Hacking--Keel--Kontsevich (GHKK) reformulated and proved the Fock--Goncharov (FG) conjecture \cite{FG06} for cluster varieties by constructing the $\vartheta$-basis. Let $(\mathcal{A},\mathcal{X})$ be a dual pair of cluster varieties. The full FG conjecture \cite[Definition 0.6]{GHKK18} asserts that the $\vartheta$-basis of the coordinate ring of $\mathcal{A}$ is parametrized by the tropical points of the dual cluster variety $\mathcal{X}$. Let $\overline{\mathcal{A}}$ denote the partial compactification of $\mathcal{A}$ obtained by allowing the frozen variables to vanish. The regular functions on $\overline{\mathcal{A}}$ are those on $\mathcal{A}$ whose orders of vanishing along the boundary divisors corresponding to the frozen variables are nonnegative. To characterize these regular functions, GHKK introduced the superpotential
$$
W \coloneqq \sum_j \vartheta_j,
$$
defined by the sum of the theta functions associated with the frozen variables. Assuming the full FG conjecture, the tropical points parametrizing the $\vartheta$-basis elements that are regular on $\overline{\mathcal{A}}$ are the integral points lying in the intersection of the half-spaces determined by the tropicalization of $W$.

Gross--Hacking--Keel--Kontsevich \cite{GHKK18} conjectured the superpotential $W$ arises from holomorphic disk counting in \cite{CO06, Aur07} as the broken lines associated with theta functions are intended to be the tropical counterparts of holomorphic disks. The main goal of this paper is to verify this conjecture for the basic affine space $G/U$, which is one of the representation-theoretic examples therein where $G \coloneqq \mathrm{SL}_{n+1}(\C)$. More precisely, we prove that the GHKK superpotential coincides with the disk potential for the basic affine space $G/U$ when restricted to a distinguished class of cluster charts of the $\mathcal{X}$-cluster variety. 

Our approach is based on relating $G/U$ to the flag variety $G/B$. The basic affine space $G/U$ plays a fundamental role in representation theory. It is a smooth quasi-affine variety whose coordinate ring decomposes into the direct sum of all irreducible representations of $G$, that is,
$$
\mathbb{C}[G/U] \simeq \bigoplus_{\lambda \in \mathsf{P}^+} V^\vee_\lambda
$$
where $\mathsf{P}^+$ is the set of dominant integral weights and $V_\lambda$ denotes the irreducible representation of highest weight $\lambda$. To recover the flag variety $G/B$, one fixes a regular dominant integral weight $\lambda$. We take $\lambda$ to be twice the sum of the fundamental weights, which corresponds to the anticanonical line bundle on $G/B$. By the Borel--Weil theorem, we have $G/B \simeq \mathrm{Proj}(\oplus_{k} V_{k \lambda}^\vee)$. Alternatively, the flag variety $G/B$ can be obtained from $G/U$ by taking a twisted GIT quotient by $B/U \simeq (\C^*)^n$ or a symplectic reduction associated with the character $\lambda$. Moreover, $G/U$ is a principal $B/U$-bundle over $G/B$. 

Both $G/U$ and $G/B$ carry cluster algebra structures. The basic affine space $G/U$ and $G/B$ contain the double Bruhat cell $G^{e, w_0}$ and the unipotent cell $U^-_{w_0}$ as open dense subsets, respectively. Their coordinate rings have upper cluster algebra structures and are related by specializing a certain collection of frozen variables. By Magee \cite{Mag15}, the full FG conjecture holds for $G/U$ and the coordinate rings $\C[G/U]$ and $\C[G^{e, w_0}]$ are isomorphic. Then $G/U$ carries the GHKK superpotential for parametrizing the $\vartheta$-basis for $\C[G/U]$.

On the other hand, the present author with Cho, Kim, and Park~\cite{CKKP25} constructed a monotone Lagrangian torus $N_{\mathsf{s}} \subseteq G/B$ for each seed $\mathsf{s}$ of the cluster algebra structure on $\C[U^-_{w_0}]$, building on the construction of Newton--Okounkov bodies, toric degenerations, completely integrable systems developed in~\cite{And13, HK15, FO25}. Since the diagonal torus $T \simeq (S^1)^n$ acts on $G/U$ in a Hamiltonian fashion, we have a principal $T $-bundle $q_\lambda \colon P \to G/B$ where $P$ is the level set corresponding to $\lambda$ of its moment map. By restricting the principal bundle $q_\lambda$ over $N_\mathsf{s}$, we obtain a Lagrangian torus $L_\mathsf{s} \coloneqq q_\lambda^{-1}(N_{\mathsf{s}})$ in $G/U$. 

The expectation is that the cluster charts of the cluster $\mathcal{X}$-variety dual to $\C[G/U]$ corresponds to the SYZ mirrors of the Lagrangian tori $\{ L_\mathsf{s} \}$ and that the GHKK superpotential restricted to the cluster chart $\mathcal{X}_\mathsf{s}$ agrees with the disk potential of $L_\mathsf{s}$.

In this paper, we prove this correspondence for a distinguished class of seeds by showing that the disk potential of $L_\mathsf{s}$ is obtained by lifting the disk potential of $N_\mathsf{s}$, see Theorems~\ref{theorem_maintheorem2} and~\ref{theorem_GHKKdisk}. For seeds arising from reduced expressions of Weyl group elements, the GHKK superpotential restricted to $\mathcal{X}_\mathsf{s}$ is explicitly written by the work of \cite{Mag15, BF19}. In \cite{CKLP23}, the present author with Cho, Lee, and Park computed the disk potential of $N_\mathsf{s}$ for a certain class of seeds. We then prove that the disk potential of $L_\mathsf{s}$ obtained by lifting holomorphic disks bounded by $N_\mathsf{s}$ agrees with the GHKK superpotential.

\begin{center}
\begin{table}[h]
\begin{tabular}{|c|c|c|c|}
\hline
& Basic affine space & $\longrightarrow$ & Flag variety \\
\hline \hline
Representations & $\bigoplus_{\lambda \in \mathsf{P}^+} V_\lambda$ & \mbox{choosing a dominant weight} & $V_\lambda$ \\
\hline
 \multirow{2}{*}{Homogeneous spaces} &  \multirow{2}{*}{$G/U$} &\text{twisted GIT quotient} & \multirow{2}{*}{$G/B$} \\ \cline{3-3}
       &   & \text{symplectic reduction} &  \\
\hline
\text{Cluster algebras} & $\C[G^{e, w_0}] $ & \text{specializing frozen variables} & $\C[U^-_{w_0}]$ \\
\hline
\hline
\text{Lagrangian tori} & $\{ L_\mathsf{s}\}$ & \text{quotient} & $\{ N_\mathsf{s}\}$ \\
\hline
\end{tabular}
\end{table}
\end{center}

For the lifting, we consider the following setup. Let $G$ be a complex reductive algebraic group that is the complexification of a compact connected Lie group $K$. Suppose that $P_G$ is a K\"{a}hler manifold and
$$
q \colon P_G \to M
$$
is a holomorphic principal $G$-bundle over a K\"{a}hler manifold $M$ obtained as the complexification of a principal $K$-bundle $q|_{P_K}\colon P_K \to M$. Let $N$ be a Lagrangian submanifold of $M$. Suppose that its preimage 
$$
L \coloneqq q |_{P_K}^{-1}(N)
$$
is a Lagrangian submanifold of $P_G$. Such a geometric setting naturally arises when a (twisted) GIT quotient is identified with the corresponding symplectic reduction by Kempf--Ness theorem \cite{KN78}, provided that the $K$-action on the level set is free.

The correspondence between holomorphic disks with boundary on $L$ in $P_G$ and those with boundary on $N$ in $M$ was established in \cite{Kim26}. In this paper, we focus on the case where $G=(\C^*)^n$ and $K=(S^1)^n$ and assume that both $N$ and $L \simeq N \times K$ are Lagrangian tori.

One of the main goals of this paper is to determine the disk potential of $L$ from the disk counting invariants of $N$. Although every holomorphic disk in $N$ admits a lift to $P_G$ via the correspondence, the boundary class of its lift contains an additional component coming from the torus fiber of the projection $q|_L \colon L \to N$. Indeed, $q|_L$ induces a decomposition
$$
\pi_1(L) \simeq \pi_1(N) \times \pi_1(K),
$$
so that the boundary class of a lifted disk has both a base component and a fiber component. While the base component is determined by the original disk in $N$, the fiber component is not. In particular, a single Laurent monomial in the disk potential of $N$ may lift to multiple distinct Laurent monomials in the disk potential of $L$.

To determine this additional fiber component, we introduce a \emph{boundary lifting lemma}, see Lemma~\ref{lemma_boundaryliftinglemma}. The key idea is to cap off this disk class into a spherical class. In general, however, the spherical class obtained by capping off a disk in the quotient does not lift to a spherical class in the total space. This discrepancy can be used to fiber components.  To measure this discrepancy, we introduce a suitable associated vector bundle and cap the lifted disks with fiberwise disk classes. It is measured by the Chern number of the associated bundle on the corresponding spherical class in the base. 

Another issue is that the total space $P_G$ is noncompact so the moduli spaces of holomorphic disks need not be compact. Under suitable positivity and regularity assumptions, we establish the compactness results and prove that the disk counting invariants of the quotient and the total space agree, see Proposition~\ref{prop_liftingmoduli}. As a consequence, the disk potential of $L$ can be computed explicitly from the disk potential of $N$.

The paper is organized as follows. 
In Section~\ref{sec_liftingholodisks}, we mainly review the results form \cite{Kim26} and compare the disk counting invariants of $L$ and $N$. In Section~\ref{sec_boundaryliftinglemma}, we introduce and prove a boundary lifting lemma. Section~\ref{sec_diskpotentialbasicaffinespaces} derive an explicit formula of the disk potential of lifted Lagrangian torus in terms of the disk counting invariants of the corresponding monotone Lagrangian torus in $G/B$. Finally, in Section~\ref{sec_GHKKstring}, we focus on the monotone Lagrangian tori constructed by the cluster algebra structure on $\C[U^{-}_{w_0}]$ and explain how the GHKK superpotential is recovered from the lifted Lagrangian tori.

\subsection*{Acknowledgement} 
The author would like to express his gratitude to Kiumars Kaveh for posing the question that inspired this paper.

\section{Lifting holomorphic disks}\label{sec_liftingholodisks}

The goal of this section is to develop a method for lifting holomorphic disks from the base of a holomorphic principal bundle to its total space, based on the results of \cite{Kim26}. In particular, we relate the disk counting invariant of a lifted class to the corresponding disk counting invariant on the base.

We briefly recall the notation for moduli spaces of holomorphic disks and disk potentials. Let $(X,\omega)$ be a symplectic manifold together with a compatible almost complex structure $J$ and let $L$ be a (relatively) spin and orientable Lagrangian submanifold. We denote by $j_0$ the standard complex structure on the unit disk $\mathbb{D} \subseteq \C$. For a relative homotopy class $\beta \in \pi_2(X, L)$, let 
$$
\mathcal{M}(X,L,J,\beta) \coloneqq \left\{ \varphi \colon (\mathbb{D}, \partial \mathbb{D}) \to (X,L) \mid d \varphi \circ j_0 = J \circ d \varphi, \, [\varphi] = \beta \right\} / \mathrm{Aut}(\mathbb{D})
$$
be the moduli space of $J$-holomorphic disks in $\beta$. We also consider the moduli space with one boundary marked point,
$$
\mathcal{M}_1(X,L,J,\beta) \coloneqq \left\{ \left( \varphi, z_0 \right) \mid \varphi \in {\mathcal{M}}(X,L,J,\beta), z_0 \in \partial \mathbb{D}  \right\} / \mathrm{Aut}(\mathbb{D}, z_0).
$$
We denote by $\mathcal{M}^{st}_1(X,L,J,\beta)$ the moduli space of stable maps in $\beta$ from a bordered Riemann surface of genus zero with one boundary marked point. 

Suppose that $\mathcal{M}_1(X,L,J,\beta)$ is a compact manifold without boundary. After choosing orientations on both $\mathcal{M}_1(X,L,J,\beta)$ and $L$, we define the disk counting invariant of $\beta$ as follows.

\begin{definition}\label{def_diskcountinginv}
The \emph{disk counting invariant} $n_\beta$ (or \emph{open Gromov--Witten invariant}) of $\beta$ is defined by the degree of the evaluation map
\begin{equation}\label{equ_degreeev}
\mathrm{ev}_0 \colon \mcal{M}_1(X,L,J,\beta) \to L, \quad [(\varphi, z_0)] \mapsto \varphi(z_0).
\end{equation}
Note that $n_\beta$ can be nonzero only if $\dim \mcal{M}_1(X,L,J,\beta) = \dim L$. 
\end{definition}

Throughout this paper, when defining disk counting invariants, we restrict our attention to the following setting. For each relative homotopy class $\beta \in \pi_2(X, L)$ of Maslov index two, suppose that
\begin{equation}\label{equ_mmst}
\mathcal{M}_1(X, L, J, \beta) = \mathcal{M}^{st}_1(X, L, J, \beta)
\end{equation} 
and
\begin{equation}\label{equ_mmcpt}
\mathcal{M}_1(X, L, J, \beta) \mbox{ is compact.} 
\end{equation}
We will simply say that $\mathcal{M}_1(X,L,J,\beta)$ is \emph{compact} when both conditions~\eqref{equ_mmst} and~\eqref{equ_mmcpt} hold. Thus, it means that every stable map in $\beta$ is a holomorphic disk and the moduli space of such disks is compact.

\begin{remark}
If $X$ is compact, then the condition~\eqref{equ_mmst} implies~\eqref{equ_mmcpt} by the Gromov compactness theorem. When $X$ is not compact, a sequence of holomorphic disks may escape to infinity and hence $\mathcal{M}_1(X, L, J, \beta)$ may not be compact.
\end{remark}

To introduce typical examples of Lagrangian submanifolds satisfying~\eqref{equ_mmst}, we recall the following notions.

\begin{definition}\label{definition_positiveconditions}
Let $(X, \omega)$ be a symplectic manifold together with a compatible almost complex structure $J$ and let $L$ be an orientable and spin Lagrangian submanifold. 
\begin{enumerate}
\item The pair $(X, J)$ is called \emph{symplectically Fano} if every nonconstant $J$-holomorphic sphere represents a class $\alpha \in \pi_2(X)$ of positive first Chern number.
\item The pair $(L, J)$ is called \emph{positive} if every nonconstant $J$-holomorphic disk represents a class $\beta \in \pi_2(X, L)$ of Maslov index at least two. 
\end{enumerate}
When the choice of $J$ is clear from the context, we simply say that $L$ is positive and $X$ is symplectically Fano. We call the pair $(X, L)$ \emph{positive} if $X$ is symplectically Fano and $L$ is positive.
\end{definition}

\begin{remark}
Suppose that the pair $(X, L)$ is positive and for every relative homotopy class $\beta \in \pi_2(X, L)$ of Maslov index two, at least one of the following conditions holds$\colon$
\begin{itemize}
\item $\partial \beta \neq 0$ in $\pi_1(L)$ whenever $\mathcal{M}_1^{st}(X, L, J, \beta) \neq \emptyset$, or
\item the minimal Chern number of $X$ is at least two.
\end{itemize}
Then~\eqref{equ_mmst} holds. 
\end{remark}

\begin{definition} 
Let $(X, \omega)$ be a symplectic manifold and let $L$ be a spin and orientable Lagrangian submanifold. Assume that $X$ is simply connected.
\begin{enumerate}
\item The symplectic manifold $(X, \omega)$ is called \emph{monotone} if there exists a constant $\lambda > 0$ such that $c_1(TX)(\alpha) =  \lambda \cdot \omega(\alpha)$ for all $\alpha \in \pi_2(X)$. 
\item The Lagrangian submanifold $L$ is called \emph{monotone} if there exists a constant $\delta > 0$ such that $\mathrm{MI}_L (\beta) = \delta \cdot \omega (\beta) $ for all $\beta \in \pi_2 (X, L)$ where $\mathrm{MI}_L(\beta)$ is the Maslov index of $\beta$.
\end{enumerate}
Note that if $X$ is monotone and $L$ is a monotone Lagrangian submanifold of $X$, then $(X,L)$ is positive.
\end{definition}

Fix an ordered basis $(\theta_j)_{j=1}^m$ for $\pi_1(L)$ represented by oriented loops and identify $\pi_1(L)$ with the lattice $\Z^m$. To each lattice point $\mathbf{v} = (v_1, \dots, v_m) \in \Z^m$, we associate the Laurent monomial 
$$
\mathbf{y}^\mathbf{v} = y_1^{v_1} \dots y_m^{v_m}. 
$$

\begin{definition}[\cite{CO06, Aur07, FOOO10}]
Suppose that every homotopy class $\beta$ of Maslov index two is regular and $\mathcal{M}_1(X, L, J, \beta)$ is compact. 
The \emph{disk potential} $W_L$ of $L$ is defined by
\begin{equation}\label{equ_potentialWL}
W_L(\mathbf{y}) = \sum_\beta n_{\beta} \, \mathbf{y}^{\partial \beta}
\end{equation}
where the summation is taken over all relative homotopy classes $\beta$.
\end{definition}

\begin{remark}
Strictly speaking, the disk potential in~\eqref{equ_potentialWL} should be defined over the Novikov ring to record the symplectic areas of holomorphic disks. In this paper, we mainly consider  monotone Lagrangian tori $L$. In this case, all holomorphic disks of Maslov index two have the same symplectic area. Thus the Novikov parameter may be specialized to a nonzero complex number and we may work over the field of complex numbers.
\end{remark}

Next, we recall several results from \cite{Kim26}. Let $G$ be a complex reductive algebraic group and let $K$ be a compact connected Lie group whose complexification is $G$. Let 
\begin{equation}\label{equ_qbundleprojection}
q \colon P_G \to M
\end{equation}
be a holomorphic principal $G$-bundle where $P_G$ and $M$ are K\"{a}hler manifolds. 
We denote by $\omega$ (resp. $\omega_M$) the K\"{a}hler form on $P_{G}$ (resp. $M$) and by $J$ (resp. $J_M$) the complex structure on $P_{G}$ (resp. $M$). 

\begin{setup}\label{assumption_pgpmathbbg}
Suppose that $P_G$ contains a submanifold $P_K$ satisfying the following conditions.
\begin{enumerate}
\item The restriction $q |_{P_K} \colon P_K \to M$ is a principal $K$-bundle.
\item The principal $G$-bundle $P_G$ is isomorphic to the extension of $P_K$ to $G$, i.e., 
$$
P_G \simeq P_K \times_K G.
$$
\item The K\"{a}hler forms satisfy $q^* \omega_{M} = \omega |_{P_K}$.
\end{enumerate}
\end{setup}

Let $L \, (\subseteq P_K)$ be a $K$-invariant, orientable Lagrangian submanifold of $P_G$. Then the restriction of $q$ induces a principal $K$-bundle
\begin{equation}\label{equ_qLprincipal}
q |_{L} \colon L \to N \coloneqq L / K.
\end{equation}
In this setting, we recall several basic properties.

\begin{lemma}[Lemmas 3.2, 3.3 and 4.12 in \cite{Kim26}]\label{lemma_factsfromKIM}
The following statements hold.
\begin{enumerate}
\item The induced homomorphism $q_* \colon \pi_2 (P_G, L) \to \pi_2(M, N)$ is an isomorphism.
\item A class $\beta \in \pi_2 (P_G, L)$ is regular if and only if $q_* \beta$ is regular. 
\item The classes $\beta$ and $q_* \beta$ have the same Maslov index.
\item The Lagrangian submanifold $N$ is positive if and only if $L$ is positive.
\end{enumerate}
\end{lemma}

Since the $K$-action on $P_G$ preserves the complex structure $J$, it induces a $K$-action on the moduli space of holomorphic disks in $P_G$ with boundary on $L$ in the class $\beta \in \pi_2(P_G, L) \colon$
$$
* \colon \mathcal{M}_1(P_{G}, L, J, \beta) \times K \to \mathcal{M}_1(P_{G}, L, J, \beta), \quad [(\varphi, z_0)] * g \mapsto  [(\varphi * g, z_0)].
$$
This action produces a principal $K$-bundle.

\begin{proposition}[Proposition 3.5 in \cite{Kim26}]\label{proposition_freemoduli}
The induced $K$-action on $\mathcal{M}_1(P_G, L, J, \beta)$ is free. In particular, if $\beta$ is regular, then the quotient $\mathcal{M}_1(P_G, L, J, \beta)/K$ is a smooth manifold and the quotient map
$$
q_* \colon \mathcal{M}_1(P_G, L, J, \beta) \to \mathcal{M}_1(P_G, L, J, \beta)/K
$$
is a principal $K$-bundle.
\end{proposition}

On the other hand, the holomorphic bundle $q$ in~\eqref{equ_qbundleprojection} induces a smooth map
$$
\widehat{\phi} \colon  \mathcal{M}_1(P_{G}, L, J, \beta) \to \mathcal{M}_1(M, N, J_M, q_* \beta), \quad [(\varphi, z_0)] \mapsto  [(q \circ \varphi, z_0)]. 
$$
Since $P_{G}$ is the extension of $P_K$, the map $\widehat{\phi}$ is $K$-invariant. Hence it descends to a smooth map
\begin{equation}\label{equ_inducedmap}
\phi \colon  \mathcal{M}_1(P_{G},L,J,\beta)/K \to \mathcal{M}_1(M, N, J_M, q_* \beta)
\end{equation}
satisfying $ \widehat{\phi} = \phi \circ q_*$.

This map $\phi$ is indeed bijective, as follows from the next two propositions.

\begin{proposition}[Proposition 3.8 in \cite{Kim26}]\label{proposition_liftingholo}
Let $\varphi \colon (\mathbb{D},\partial \mathbb{D}) \to (M,N)$ be a $J_M$-holomorphic disk. Then $\varphi$ admits a holomorphic lift. Namely, there exists a $J$-holomorphic disk $\varphi^\sharp \colon (\mathbb{D},\partial \mathbb{D}) \to (P_{G},L)$ such that $q \circ \varphi^\sharp = \varphi$.

In particular, the map $\phi$ is surjective.
\end{proposition}

\begin{proposition}[Lemma 3.12 in \cite{Kim26}]\label{proposition_injective}
The map $\phi$ is injective.
\end{proposition}

We now obtain the holomorphic disk correspondence under Setup~\ref{assumption_pgpmathbbg}.

\begin{theorem}[Theorem 3.13 and Corollary 3.14 in \cite{Kim26}] \label{theorem_fundacycle}
Suppose that the quotient space $N = L / K$ is spin, $\beta \in \pi_2(P_{G}, L)$ is regular, $\mathcal{M}_1(P_G, L, J, \beta)$ is compact. If $L$ is equipped with the product spin structure, then the map
$$
\phi \colon \mathcal{M}_1(P_{G},L,J,\beta)/K \to \mathcal{M}_1(M, N, J_M, q_* \beta), \quad [(\varphi, z_0)] \mapsto  [(q \circ \varphi, z_0)] 
$$
is an orientation-preserving homeomorphism. 

In particular, the induced homomorphism $\phi_*$ maps the fundamental class of $\mathcal{M}_1(P_{G},L,J,\beta)/K$ to that of $\mathcal{M}_1(M, N, J_M, q_* \beta)$. Moreover, the disk counting invariants of $\beta$ and $q_* \beta$ coincide$\colon$
$$
n_\beta = n_{q_* \beta}.
$$
\end{theorem}

Using this correspondence, we study holomorphic disks in the total space via those in the base. Let $\beta \in \pi_2(M,N)$. Since the homomorphism $q_*$ is an isomorphism by Lemma~\ref{lemma_factsfromKIM}, we define the lift of $\beta$ by 
\begin{equation}\label{equ_liftedclass}
\beta^\sharp \coloneqq q_*^{-1} (\beta) \in \pi_2(P_G, L).
\end{equation}
We next establish the compactness of the moduli space of holomorphic disks representing $\beta^\sharp$. 

\begin{proposition}\label{prop_liftingmoduli}
Suppose that $\beta \in \pi_2(M,N)$ is regular. If the moduli space $\mathcal{M}_1(M, N, J_M, \beta)$ is compact, then the moduli space $\mathcal{M}_1(P_G,L,J,\beta^\sharp)$ is also compact. In particular, 
\begin{equation}\label{equ_countinginvarcoin}
n_{\beta} = n_{\beta^\sharp}.
\end{equation}
\end{proposition}

\begin{proof} 
By Lemma~\ref{lemma_factsfromKIM}, the lifted class $\beta^\sharp$ is regular. By Proposition~\ref{proposition_freemoduli}, $\mathcal{M}_1(P_{G},L,J,\beta^\sharp) / K$ is a smooth manifold whose dimension is equal to $\dim \mathcal{M}_1(M, N, J_M, \beta)$. Moreover, by Propositions~\ref{proposition_liftingholo} and~\ref{proposition_injective}, the induced map $\phi$ in~\eqref{equ_inducedmap} is a smooth bijection. Since the domain and codomain have the same dimension, by invariance of domain, $\phi$ is an open map. Therefore, $\phi$ is a homeomorphism. Since $K$ is compact and $\mathcal{M}_1(M, N, J_M, \beta)$ is compact by assumption, it follows that $\mathcal{M}_1(P_G, L, J, \beta^\sharp)/K$ and $\mathcal{M}_1(P_G, L, J, \beta^\sharp)$ are compact. The equality~\eqref{equ_countinginvarcoin} then follows from Theorem~\ref{theorem_fundacycle}.
\end{proof}

Next, we show that only holomorphic disks of Maslov index two (without bubbling) contribute  to the disk potential under certain positivity and topological assumptions. Hence, it suffices to consider lifted holomorphic disks to compute the disk potential of $L$ in $P_G$. 

\begin{proposition}\label{proposition_maina}
Suppose that $\pi_2(P_G) \simeq \{0\}$. Let $\beta^\sharp \in \pi_2(P_G,L)$ be a class of Maslov index two. If $L$ is positive, then every stable map from a bordered Riemann surface of genus zero with one boundary marked point in $\beta^\sharp$ is a holomorphic disk, that is, 
$$
\mathcal{M}_1(P_G, L, J, \beta^\sharp) = \mathcal{M}^{st}_1(P_G, L, J, \beta^\sharp).
$$
\end{proposition}

\begin{proof}
Let $\varphi \colon (\Sigma, \partial \Sigma) \to (P_G, L)$ be a stable map in the class $\beta^\sharp$. Since $\pi_2(P_G) \simeq \{0\}$, the domain $\Sigma$ contains no sphere components. Moreover, since $L$ is positive, every nonconstant disk component has Maslov index at least two. Since $\beta^\sharp$ has Maslov index two, there is exactly one nonconstant disk component and thus every stable map in $\beta^\sharp$ is a holomorphic disk.
\end{proof}

\begin{corollary}\label{cor_diskposreg}
Suppose that every class $\beta \in \pi_2(M,N)$ of Maslov index two is regular and $\mathcal{M}_1(M,N,J_M,\beta)$ is compact. If $N$ is positive and $\pi_2(P_G) \simeq \{0\}$, then the disk potential of $L$ is given by
$$
W_L (\mathbf{y}) = \sum_{\beta \in \pi_2(M,N)} n_\beta \, \mathbf{y}^{\partial \beta^\sharp}.
$$
\end{corollary}

\begin{proof}
Let $\beta \in \pi_2(M,N)$ be a homotopy class of Maslov index two. By Lemma~\ref{lemma_factsfromKIM}, the lifted class $\beta^\sharp$ is also of Maslov index two and is regular, and $L$ is positive. By Proposition~\ref{proposition_maina}, 
$$
\mathcal{M}_1(P_G, L, J, \beta^\sharp) = \mathcal{M}^{st}_1(P_G, L, J, \beta^\sharp).
$$ 
Thus, only holomorphic disks contribute to $W_L$. Since $q_* (\beta^\sharp) = \beta$ and $\mathcal{M}_1(M, N, J_M, \beta)$ is compact by assumption, by Proposition~\ref{prop_liftingmoduli}, we have $n_\beta = n_{\beta^\sharp}$, which proves the formula.
\end{proof}

\begin{remark}
The relationship between Maurer--Cartan spaces and disk potentials on a symplectic manifold equipped with a Hamiltonian $K$-action and those on its symplectic reduction was explored in Lau--Leung--Li \cite{LLL23} via equivariant Lagrangian correspondences. The above relation is expected to be derived by showing that the obstruction in Lagrangian correspondence vanishes.
\end{remark}

\section{Boundary lifting lemma}\label{sec_boundaryliftinglemma}

In this section, we establish a boundary lifting lemma that determines the boundary classes of the lifted holomorphic disks in Proposition~\ref{proposition_liftingholo}. This lemma will play a key role in computing the disk potential from the lifted holomorphic disks.

Throughout this section, we assume Setup~\ref{assumption_pgpmathbbg} and the following additional hypotheses$\colon$
\begin{enumerate}
\item $L \simeq T^{m+n}$, $K \simeq T^n$, and $N \simeq T^{m}$ are tori and
\item the base manifold $M$ is simply connected. 
\end{enumerate}
Under these assumptions, we have
\begin{align}
\label{equ_splittingH2rel}&H_2(M, N; \Z) \simeq \pi_2(M, N) \simeq \pi_2(M) \oplus \pi_1(N) \simeq H_2(M; \Z) \oplus H_1(N; \Z), \\
\label{equ_splittingH1}&H_1(L; \Z) \simeq \pi_1(L) \simeq \pi_1(N) \oplus \pi_1(K) \simeq H_1(N; \Z) \oplus H_1(K; \Z).
\end{align}

To motivate the boundary lifting problem, observe that the induced homomorphism
\begin{equation}\label{equ_inducedhomoH1}
q |_{L,*} \colon \pi_1(L) \to \pi_1(N)
\end{equation}
has a nontrivial kernel. Thus, even if two relative homotopy classes $\beta_1$ and $\beta_2$ satisfy
$$
\partial \beta_1 = \partial \beta_2 \quad \mbox{in $\pi_1(N)$},
$$
their lifts $\beta_1^\sharp$ and $\beta_2^\sharp \in \pi_2(P_G,L)$ need not have the same boundary. In other words, a single monomial in the disk potential $W_N$ of $N$ may lift to a sum of \emph{distinct} monomials in the disk potential $W_L$ of $L$. 

Now let $\varphi \colon (\mathbb{D}, \partial \mathbb{D}) \to (M, N)$ be a holomorphic disk. By Proposition~\ref{proposition_liftingholo}, $\varphi$  admits a holomorphic lift 
$$
{\varphi^\sharp} \colon (\mathbb{D}, \partial \mathbb{D}) \to (P_{G}, L).
$$ 
Since $q \circ \partial {\varphi^\sharp} = \partial {\varphi}$, the boundary class $[\partial \varphi^\sharp]$ is determined by $[\partial \varphi]$ up to an element of $\pi_1(K)$ under the decomposition~\eqref{equ_splittingH1}.

To determine this $\pi_1(K)$-component, we fix a faithful representation $\rho \colon {G} \to \mathrm{GL}(V)$ where $G \simeq (\C^*)^n$ and $V \simeq \C^n$. We let  $G$ act on $P_G \times V$ from the right by 
\begin{equation}\label{equ_convention}
(p, v) * g = (pg, \rho(g) v) \quad \mbox{for $p \in P_G$, $v \in V$, and $g \in G$.}
\end{equation}
and consider the associated vector bundle 
$$
q_\mathcal{E} \colon \mathcal{E} \coloneqq P_{G} \times_\rho V \to M.
$$

Since $G$ is abelian, every complex irreducible representation of $G$ is one-dimensional. Moreover, the representation $V$ is completely reducible and hence decomposes as 
\begin{equation}\label{equ_decompositionvi}
V = V_1 \oplus \cdots \oplus V_n
\end{equation}
where each $V_j$ is one-dimensional.  Accordingly, the representation $\rho$ decomposes into characters $\rho = \oplus_{j=1}^n \chi_j$. For each $j$, let
\begin{equation}\label{equ_linebundleasso}
\mathcal{L}_j \coloneqq {P}_{G} \times_{\chi_j} V_j
\end{equation}
denote the associated line bundle. Then 
\begin{equation}\label{equ_bundleproj}
\mathcal{E} = {P}_{G} \times_\rho V \simeq \bigoplus_{j=1}^n \, ({P}_{G} \times_{\chi_j} V_j) \simeq \bigoplus_{j=1}^n \mathcal{L}_j.
\end{equation}

\begin{remark}\label{remark_conventions}
Under the convention~\eqref{equ_convention}, the equivalence relation for the associated vector bundle is
\begin{equation}\label{equ_equivalencerelations}
[pg,v] = [p,\rho(g)^{-1}v].
\end{equation}
We use this convention to ensure that the homogeneous line bundle $\mathcal{L}_\lambda$ over $G/B$ is ample for every regular dominant weight $\lambda$.
\end{remark}

Let $v_j$ be a nonzero vector in $V_j$ and set $v \coloneqq \sum v_j$. Since $\rho$ is faithful, the stabilizer of $v$ is trivial. Then the map
$$
\iota \colon {P}_{G} \to \mathcal{E}, \,\quad \mbox{ $p \mapsto [(p, v)]$.}
$$
is an open embedding. 
Moreover, the following diagram commutes$\colon$ 
\begin{equation}
\xymatrix{
{P}_{G} \ar[rr]^{\iota} \ar[rd]_{q} & & \mathcal{E} \ar[ld]^{q_\mathcal{E}}\\
& M \simeq {P}_{G} \git {G} & 
}
\end{equation}
In particular, $\iota(L)$ is an $(m+n)$-dimensional torus invariant under the fiberwise $G$-action on $\mathcal{E}$. By abuse of notation, we often denote $\iota(L)$ by $L$. In particular, we have the following lemma.

\begin{lemma}\label{lemma_Lisomorphic}
The map $\iota$ induces an isomorphism
$$
\iota_* \colon \pi_1(L) \to \pi_1(\iota(L)).
$$ 
\end{lemma}

For computational purposes, it is convenient to work with absolute homotopy classes rather than relative homotopy classes. We fix a collection of relative homotopy classes that can be capped off to form spherical classes. More precisely, we take an $m$-tuple of homotopy classes
\begin{equation}\label{equ_mhomotopyboundary}
\beta_1, \dots, \beta_m \in \pi_2(M, N) 
\end{equation}
such that $\partial \beta_1, \dots, \partial \beta_m$ form a basis of $\pi_1(N) \simeq H_1(N; \Z) \simeq \Z^m$.

For a given class $\beta \in \pi_2(M,N)$, there exists a unique integral vector 
\begin{equation}\label{equ_mathbbb}
\mathbf{b} \coloneqq (b_1, \dots, b_m) \in \Z^m
\end{equation}
such that
\begin{equation}\label{equ_boundaryfree}
\partial \beta + \sum_{j=1}^m b_j \partial \beta_j = 0.
\end{equation}
Under the identification~\eqref{equ_splittingH2rel}, the class
\begin{equation}\label{equ_sphericalalpha}
\alpha \coloneqq \beta + \sum_{j=1}^m b_j \beta_j
\end{equation}
may be regarded as an element of $\pi_2(M)$.

Now define
\begin{equation}\label{equ_alphasharp}
\alpha^\sharp \coloneqq \beta^\sharp + \sum_{j=1}^m b_j \beta^\sharp_j.
\end{equation}
In contrast to $\alpha$, the lifted class $\alpha^\sharp$ need \emph{not} be spherical since $q_* \colon \pi_2(P_G) \to \pi_2(M)$ is not surjective in general. Because $q_* (\alpha^\sharp) = \alpha$, the boundary class $\partial \alpha^\sharp$ has trivial $\pi_1(N)$-component. Hence, under the decomposition~\eqref{equ_splittingH1}, it is determined by its $\pi_1(K)$-component. By Lemma~\ref{lemma_Lisomorphic}, it suffices to determine the boundary class of $\iota_*(\alpha^\sharp)$ in $\mathcal{E}$. For simplicity, we continue to denote $\iota_*(\alpha^\sharp)$ by $\alpha^\sharp$.

For each $j = 1, \dots, n$, let $\eta_j \in \pi_2(\mathcal{E},L)$ denote the relative homotopy class represented by a fiberwise disk in the line bundle $\mathcal{L}_j$. Its boundary is an $S^1$-orbit in the fiber and it intersects the zero section of $\mathcal{L}_j$ transversely and positively at a single point. By adding suitable multiples of the classes $\eta_j$ to $\alpha^\sharp$, we can cap off $\alpha^\sharp$ to obtain a spherical class. Set
\begin{equation}\label{equ_widetildealpha1}
\widetilde{\alpha} = \beta^\sharp + \sum_{j=1}^m b_j \beta_j^\sharp + \sum_{j=1}^n a_j \eta_j
\end{equation}
where $a_1, \dots ,a_n \in \Z$. The following lemma determines these coefficients explicitly.

\begin{lemma}[Boundary lifting lemma]\label{lemma_boundaryliftinglemma}
For each $j = 1, \dots, n$, let $c_1(\mathcal{L}_j)$ denote the first Chern class of the line bundle $\mathcal{L}_j$ defined in~\eqref{equ_linebundleasso}. If we choose
$$
a_j \coloneqq \langle c_1(\mathcal{L}_j), \alpha \rangle,
$$
then the class $\widetilde{\alpha}$ in~\eqref{equ_widetildealpha1} is spherical, that is, 
\begin{equation}\label{equ_widetildealpha}
\partial \widetilde{\alpha} = 0.
\end{equation}
\end{lemma}

\begin{proof}
Let $\mathrm{pr}_j \colon \mathcal{E} \to \mathcal{L}_j$ denote the projection onto the $j$-th summand in the decomposition~\eqref{equ_bundleproj}.
$\mathrm{pr}_j (\eta_i)$ is trivial for $i \neq j$, the coefficient $a_j$ can be determined independently by projecting to the line bundle $\mathcal L_j$.

Let $s_j \colon M \to \mathcal{L}_j$ be the zero section and set $Z_j \coloneqq s_j(M)$. Recall that the Poincar\'{e} dual of $Z_j$ is the Thom class $U_j$ of the line bundle $\mathcal{L}_j$. Moreover, the Euler class is given by 
$$
e(\mathcal{L}^\R_{j}) = s_j^* (U_j),
$$
see in \cite[Section 6]{BT82} for instance.  

Set
$$
\widetilde{\alpha}_j \coloneqq \mathrm{pr}_{j,*} \widetilde{\alpha} \,\, \mbox{ and } \,\, q_j \colon \mathcal{L}_j \to M.
$$
Then we have $\alpha = q_{j,*} (\widetilde{\alpha}_j)$ and
\begin{equation}\label{equ_c1euler}
\langle c_1(\mathcal{L}_j), \alpha \rangle = \langle e(\mathcal{L}^\R_{j}), q_{j,*} (\widetilde{\alpha}_j) \rangle = 
\langle s_j^* (U_{j}),  q_{j,*} (\widetilde{\alpha}_j) \rangle = 
\langle  U_j, \widetilde{\alpha}_j \rangle
\end{equation} 
since $s_j \circ q_j$ and $\mathrm{id}$ are homotopic. The right hand side is the algebraic intersection number of the homology class $\widetilde{\alpha}_j$ and the zero section $Z_j$. Because the lifted disk classes $\beta^\sharp$ and $\beta_j^\sharp$ are contained in $P_G$, they are disjoint from $Z_j$. Thus, the only contribution to the intersection number comes from the term $a_j \eta_j$ and hence~\eqref{equ_c1euler} is equal to $a_j$.
\end{proof}

Let $(\theta_i)_{i=1}^{m+n}$ be an ordered basis for $\pi_1(L) \simeq \Z^{m+n}$ represented by oriented loops. Since the restriction $q |_L$ in~\eqref{equ_qLprincipal} is a principal $K$-bundle, the induced homomorphism $q |_{L,*} \colon \pi_1(L) \to \pi_1(N)$ is onto. Moreover,
$$
\pi_1(L) / \mathrm{ker}( q |_{L,*} ) \simeq \pi_1(N) \simeq \Z^m
$$
so the kernel of $q |_{L,*}$ is a primitive sublattice of $\pi_1(L) \simeq \Z^{m+n}$. 

After replacing the basis if necessary, we may assume the following. 
\begin{enumerate}
\item For each $i = 1, \dots, n$, 
$$
\theta_{m+i} \coloneqq \partial \eta_i
$$  
where $\eta_i$ is given in~\eqref{equ_widetildealpha1}. Consequently, $(\theta_i)_{i={m+1}}^{m+n}$ forms a basis of $\mathrm{ker}(q |_{L,*}) \simeq \pi_1(K)$.
\item The tuple $(q_{L} \circ \theta_j)_ {j=1}^m$ forms a basis of $\pi_1(N)$.
\end{enumerate}

Let $(e_i)_{i=1}^{m+n}$ denote the basis of $H^1(L; \Z)$ dual to $(\theta_i)_{i=1}^{m+n}$. This induces a basis $(f_i)_{i=1}^m$ of $H^1(N ; \Z)$ such that $(q |_L)^* f_i = e_i$ for $i = 1, \dots, m$. Let $\mathbf{y} = (y_i)$ (resp. $\mathbf{z} = (z_i)$) be the variables corresponding to $\exp(e_i)$ for $L$ (resp. $\exp(f_i)$ for $N$). Then the disk potentials $W_L$ and $W_N$ are Laurent polynomials in the variables $\mathbf{y}$ and $\mathbf{z}$, respectively.

From~\eqref{equ_boundaryfree}, it follows that 
$$
\mathbf{z}^{\partial \beta}  \cdot \prod_{j=1}^m \mathbf{z}^{b_j \partial \beta_j}= 1.
$$
Similarly, by~\eqref{equ_widetildealpha}, we have
$$
\mathbf{y}^{\partial \beta^\sharp} \cdot \prod_{j=1}^m \mathbf{y}^{b_j \partial \beta_j^\sharp} \cdot \prod_{j=1}^n y_{m+j}^{a_j} = 1.
$$

The above discussion is summarized as follows.

\begin{corollary}
For every $\beta \in \pi_2(M,N)$,
$$
\mathbf{y}^{\partial \beta^\sharp}  = \prod_{j=1}^m \mathbf{y}^{- b_j\partial \beta_j^\sharp} \cdot \prod_{j=1}^n y_{m+j}^{- a_j}.
$$
\end{corollary}

We close this section with an elementary example adapted from \cite{Kim26}.

\begin{example}
Let $n \in \mathbb{Z}_{\geq 0}$ and let $\mathcal{L} = \mathcal{O}(-n)$ be the line bundle over $\CP^1$. Let $P$ denote the complement of the zero section of $\mathcal{L}$. Then $P$ is a principal $\C^*$-bundle over $\CP^1$ with respect to the fiberwise $\C^*$-action. 

The induced $S^1$-action admits a symplectic reduction $P \git S^1 \simeq \CP^1$. Let $N \subseteq \CP^1$ be an equator, regarded as a Lagrangian submanifold that divides $\CP^1 \simeq S^2$ into two hemispheres. There are precisely two holomorphic disks of Maslov index two, one passing through the north pole and the other passing through the south pole. Let $\beta_1, \beta_2 \in\pi_2(\CP^1,N)$ denote their relative homotopy classes. The disk potential of $N$ is 
$$
W_{N}(z) = z + \frac{1}{z}.
$$

Let $L (\subseteq P)$ be the principal $S^1$-bundle over $N$ contained in a level set of the moment map for the fiberwise $S^1$-action. Consider the associated line bundle $\mathcal{E} \to \CP^1$ corresponding to the representation 
$$
\rho \colon \C^* \to \mathrm{GL}(\C)\,\, \mbox{ given by } \, (g \mapsto (v \mapsto g^{-1}v)).
$$ 
Then $\mathcal{E}$ is isomorphic to $\mathcal{O}(-n)$, cf. Remark~\ref{remark_conventions}. 

Let $\beta_1^\sharp, \beta_2^\sharp \in \pi_2(P,L)$ denote the lifted holomorphic disk classes. Choose 
\begin{equation}\label{equ_choiceofvartheta}
\theta_1 \coloneqq \partial \beta_1^\sharp \mbox{ and } \theta_2 \coloneqq \partial \eta,
\end{equation}
where $\eta$ is the relative homotopy class of a fiberwise disk bounded by an $S^1$-orbit and intersecting the zero section transversely and positively at a single point. Let $y_1$ and $y_2$ be the corresponding exponential variables. By Proposition~\ref{prop_liftingmoduli}, the disk potential of $L$ in $P$ is 
$$
W_{L}(\mathbf{y}) = \mathbf{y}^{\partial \beta_1^\sharp} + \mathbf{y}^{\partial \beta_2^\sharp}.
$$

Since $\langle c_1(\mathcal{E}), [\beta_1 + \beta_2] \rangle = - n$,
the boundary lifting lemma yields
$$
\mathbf{y}^{\partial \beta_1^\sharp} \mathbf{y}^{\partial \beta_2^\sharp} y_2^{-n} = 1.
$$
By our choice of basis in~\eqref{equ_choiceofvartheta},
$$
\mathbf{y}^{\partial \beta_1^\sharp} = y_1 \mbox{ and } \mathbf{y}^{\partial \beta_2^\sharp} = \frac{y_2^n}{y_1}.
$$
Therefore,
\begin{equation}\label{equ_diskpotentialLO(-n)}
W_{L}(\mathbf{y}) = y_1 + \frac{y_2^n}{y_1}.
\end{equation}
\end{example}

\begin{remark}
If we instead use the standard representation $\rho(g) v = g v$, then $\mathcal{E} \simeq \mathcal{O}(n)$ and the corresponding disk potential is
$$
W_L(\mathbf{y}) = y_1 + \frac{1}{y_1y_2^n},
$$
which is equal to~\eqref{equ_diskpotentialLO(-n)} up to a coordinate change. 
\end{remark}

\section{Disk potentials for basic affine spaces}\label{sec_diskpotentialbasicaffinespaces}

The aim of this section is to compute the disk potential for the basic affine space from the disk counting invariants of a monotone Lagrangian torus in the flag variety, using the machinery developed in Sections~\ref{sec_liftingholodisks} and~\ref{sec_boundaryliftinglemma}. 

\subsection{Flag manifolds and basic affine spaces}

We begin by reviewing the basic affine space and the flag variety, fixing notation and recalling the facts needed later.

Let $G \coloneqq \mathrm{SL}_{n+1}(\C)$ be the simply connected complex simple Lie group of type $A_n$ with Lie algebra $\mathfrak{g}$. Set $[n] \coloneqq \{1, 2, \dots, n\}$. Let $\mathsf{A} = (a_{i,j})_{i,j\in [n]}$ be the Cartan matrix of $\mathfrak{g}$. Let $H$ be the maximal torus of $G$ consisting of all diagonal matrices of $G$ and let $B$ (resp. $B^-$) be the Borel subgroup consisting of all upper (resp. lower) triangular matrices of $G$. We denote by $U$ (resp. $U^-$) the unipotent radical of $B$ (resp. $B^-$).

Let $\mathfrak{h}$ be the Lie algebra of $H$ and let $\mathfrak{h}^* \coloneqq \mathrm{Hom}_\mathbb{C}(\mathfrak{h}, \mathbb{C})$ be its dual. We write 
$
\langle- , -  \rangle \colon \mathfrak{h} \times \mathfrak{h}^* \to \C 
$ for the natural pairing. Let $\Phi \coloneqq \Phi(G,H)$ be the root system and let $\Phi^+ \coloneqq \Phi^+(B,H)$ be the set of positive roots. The Weyl group $W$ of $G$ is defined by $N_G(H)/H$ where $N_G(H)$ is the normalizer of $H$ in $G$. It is isomorphic to the symmetric group $\mathfrak{S}_{n+1}$.

For each $i \in [n]$, let $\alpha_i \in \mathfrak{h}^*$ be the $i$-th simple root, let $s_i \in W$ be the corresponding simple reflection, and let $h_i \in \mathfrak{h}$ be the corresponding simple coroot defined by $\langle h_i, \alpha_j \rangle = a_{i,j}$ for all $j \in [n]$. The $i$-th fundamental weight $\varpi_i \in \mathfrak{h}^*$ is defined by $\langle h_j, \varpi_i \rangle = \delta_{j,i}$ for all $j \in [n]$. The weight lattice $\mathsf{P}$ is 
$$
\mathsf{P} \coloneqq \bigoplus_{i \in [n]} \Z \varpi_i.
$$
The sets of dominant and regular dominant integral weights are defined by
\begin{align*}
&\mathsf{P}^+ \coloneqq \{ \lambda \in \mathsf{P} \mid \langle h_i, \lambda  \rangle \geq 0 \mbox{ for all } i \in [n] \} \mbox{ and } \\
&\mathsf{P}^{++} \coloneqq \{ \lambda \in \mathsf{P} \mid \langle h_i, \lambda  \rangle > 0 \mbox{ for all } i \in [n] \},
\end{align*}
respectively.

The complete flag manifold is 
$$
\mathcal{F}\ell(n+1) \coloneqq \bigl\{ (0 \subseteq V_1 \subseteq \dots \subseteq V_n \subseteq \C^{n+1} ) \mid \dim_\C V_j = j \mbox{ for each $j \in [n]$} \bigr\}.
$$
It admits the Pl\"{u}cker embedding to the product of projective spaces 
\begin{equation}\label{equ_prodofproj}
\prod_{j=1}^n \mathbb{P}(V_{\varpi_j}) \simeq \prod_{j=1}^n \mathbb{P} \bigl(\wedge^j \C^{n+1}\bigr)
\end{equation}
where $V_{\varpi_j}$ is the irreducible representation of highest weight $\varpi_j$. Explicitly, if $\{v_{i1},\dots,v_{ii}\}$ is a basis of $V_i$ for each $i \in [n]$, then the embedding is given by
$$
\bigl(0 \subseteq V_1 \subseteq \dots \subseteq V_n \subseteq \C^{n+1} \bigr) \mapsto \bigl( [v_{11}], [v_{21}\wedge v_{22}], \dots, [v_{n1}\wedge\cdots\wedge v_{nn}] \bigr).
$$
The homogeneous coordinates on~\eqref{equ_prodofproj} are called the Pl\"{u}cker coordinates. The image of the complete flag manifold $\mathcal{F}\ell(n+1)$ under the Pl\"{u}cker embedding is defined by the Pl\"{u}cker relations. The group $G$ acts transitively on $\mathcal{F}\ell(n+1)$ and the stabilizer of the standard flag is the Borel subgroup $B$. Hence,
$$
\mathcal{F}\ell(n+1) \simeq G/B,
$$
which is a projective homogeneous variety of complex dimension $m \coloneqq {n(n+1)}/{2}$. 

A regular dominant integral weight $\lambda \in \mathsf{P}^{++}$ can be written as
\begin{equation}\label{equ_lambdagiven}
\lambda = \sum_{j=1}^{n} \eta_j \varpi_j \mbox{ with $\eta_j > 0$ for all $j$.}
\end{equation}
Let $\epsilon_j \in \mathfrak{h}^*$ be the $j$-th standard weight defined by $A \mapsto a_{jj}$. 
Setting $\lambda_j \coloneqq \sum_{i=j}^{n} \eta_i$ and $\lambda_{n+1} = 0$, we obtain 
$$
\lambda = \sum_{j=1}^{n+1} \lambda_j  \epsilon_j.
$$
We equip $\mathcal{F}\ell(n+1)$ with the K\"{a}hler form obtained by pulling back the product K\"{a}hler form
\begin{equation}\label{equ_FubiniStudy}
\sum_{j=1}^n (\lambda_j - \lambda_{j+1}) \, \omega_{\mathbb{P}(V_{\varpi_j})}
\end{equation}
where $\omega_{\mathbb{P}(V_{\varpi_j})}$ denotes the Fubini--Study form normalized so that $\int_{\CP^1} \omega_{\mathbb{P}(V_{\varpi_j})} = 1$. We denote the resulting K\"{a}hler form on $\mathcal{F}\ell(n+1)$ by $\omega_\lambda$. 

The weight $\lambda$ also determines the homogeneous line bundle over $G/B$
$$
\mathcal{L}_\lambda \coloneqq (G \times \C) / B \quad \mbox{where } (g,v) * b = (gb, \lambda(b) v) \,\, \mbox{for $g \in G, b \in B$, and $v \in \C$}.
$$ 
If $\lambda  \in \mathsf{P}^{++}$, then $\mathcal{L}_\lambda$ is very ample. Moreover, the K\"{a}hler form $\omega_\lambda$ represents the first Chern class of $\mathcal{L}_\lambda$, namely, 
\begin{equation}\label{equ_c1w}
c_1(\mathcal{L}_\lambda) = [\omega_\lambda] \quad \mbox{in $H^2(\mathcal{F}\ell(n+1) ; \Z)$.}
\end{equation}

For each Weyl group element $w \in W$, we choose a representative in $N_G(H)$, which we also denote by $w$ by abuse of notation. The Schubert cell associated with $w$ is defined by $C_w \coloneqq B{w}B/B$ and the Schubert variety $X_{w}$ is defined by its Zariski closure in $G/B$. The homology $H_\bullet(G/B;\Z)$ has a basis consisting of the fundamental classes of Schubert varieties. In particular, for each simple reflection $s_i$, the Schubert variety $X_{s_i}$ is a complex curve and the classes $[X_{s_i}]$ for $i \in [n]$ form a basis of $H_2(G/B;\Z)$.

We now introduce the basic affine space. Consider the affine space 
\begin{equation}\label{equ_prodaffinespace}
\prod_{j=1}^n \bigl(\wedge^j \C^{n+1}\bigr)
\end{equation}
whose coordinates are called the (affine) Pl\"{u}cker coordinates. The Pl\"{u}cker relations define an affine subvariety of~\eqref{equ_prodaffinespace}. In addition, we require each component $\xi_j \in \wedge^j \C^{n+1}$ to be nonzero and then we obtain a smooth quasi-affine subvariety, denoted by $\mathcal{B}(n+1)$ and called the \emph{basic affine space} or \emph{fundamental affine space}. It is equipped with the K\"{a}hler form $\omega_{\mathcal{B}(n+1)}$ induced from the standard symplectic form on the ambient affine space. 

The group $G$ acts transitively on $\mathcal{B}(n+1)$ and the stabilizer of the standard point is the unipotent radical $U$ of $B$. Hence, 
$$
\mathcal{B}(n+1) \simeq G/U. 
$$
Moreover, $G/U$ carries a natural left action of $G$ and a natural right action of $B/U$. The quotient by the right action yields the projection
\begin{equation}\label{equ_projectionq}
q \colon G/U \to G/B,
\end{equation}
which is a principal $B/U$-bundle. In the next subsection, we use this bundle to relate the disk counting invariants of $G/U$ and $G/B$.

\subsection{Disk potentials for basic affine spaces}

Let $G = \mathrm{SL}_{n+1}(\C)$ and let $N$ be a Lagrangian torus in the flag manifold $G/B$ where $\dim_\C G/B = m = n(n+1)/2$. The goal of this subsection is to compute the disk potential for $G/U$ by lifting holomorphic disks from $G/B$ and relating their open Gromov--Witten invariants to those of $N$.

Consider the principal $B/U$-bundle $q \colon G/U \to G/B$ defined in~\eqref{equ_projectionq}. We identify $B/U$ with $(\C^*)^n$ via
\begin{equation}\label{equ_identification}
B/U \simeq (\C^*)^{n+1}/ \mathrm{diag}(\C^*) \simeq (\C^*)^n, \quad [(b_{i,j})] \mapsto \bigl(b_{1,1} b_{2,2}^{-1}, b_{2,2} b^{-1}_{3,3}, \dots, b_{n,n}b^{-1}_{n+1,n+1} \bigr).
\end{equation}
The torus $(\C^*)^n$ acts on $\prod_{j=1}^n \bigl(\wedge^j \C^{n+1}\bigr)$ from the right by coordinatewise scaling$\colon$
\begin{equation}\label{equ_actionT}
(v_1, \dots, v_n) * (t_1, \dots, t_n) \mapsto (t_1v_1, \dots, t_nv_n).
\end{equation}
Under the identification~\eqref{equ_identification}, the right $B/U$-action on $G/U$ is the restriction of the coordinatewise scaling action~\eqref{equ_actionT} to the quasi-affine subvariety $\mathcal{B}(n+1) \simeq G/U$.

Let $T \coloneqq (S^1)^n$ denote the maximal compact subgroup of $(\C^*)^n$. The induced $T$-action on $G/U$ is Hamiltonian and we denote its moment map by 
$$
\mu_T \colon G/U \to \R^n.
$$ 

For the regular dominant integral weight $\lambda \in \mathsf{P}^{++}$ in~\eqref{equ_lambdagiven}, define
\begin{equation}\label{equ_choiceofnu}
\nu \coloneqq (\eta_1, \dots, \eta_n) = (\lambda_1-\lambda_2, \dots , \lambda_n-\lambda_{n+1}).
\end{equation}
The restriction of $q$ to the level set $\mu_T^{-1}(\nu)$
\begin{equation}\label{equ_quotientmapatlevelset}
q_\nu \coloneqq q|_{\mu_T^{-1}(\nu) } \colon \mu_T^{-1}(\nu) \to G/B
\end{equation}
is a principal $T$-bundle. Moreover, by the choice of the level set, the K\"{a}hler forms satisfy
\begin{equation}\label{equ_Kahlerforms}
q_\nu^* \, \omega_{\lambda} = \omega_{\mathcal{B}(n+1)}\big|_{\mu_T^{-1}(\nu)}.
\end{equation}

Let $L$ be the union of the $T$-orbits over the Lagrangian torus $N \subseteq G/B$ inside the level set $\mu_T^{-1}(\nu)$. Equivalently, $L = q_\nu^{-1}(N)$. Since $q_\nu$ is a principal $T$-bundle satisfying~\eqref{equ_Kahlerforms}, it follows that $L$ is a Lagrangian submanifold of $G/U$. Throughout this subsection, we assume that $L$ is a Lagrangian torus in $G/U$, cf. Proposition~\ref{proposition_toricdegtorus}.

Fix homotopy classes 
\begin{equation}\label{equ_beta1m}
\beta_1, \dots, \beta_m \in \pi_2(G/B, N) \simeq \pi_2(G/B) \times \pi_1(N)
\end{equation}
such that $\{\partial \beta_1, \dots, \partial \beta_m\}$ forms a basis of $\pi_1(N) \simeq \Z^m$ as in~\eqref{equ_mhomotopyboundary}. Then, for every class $\beta \in \pi_2(G/B,N)$, there exists a unique integral vector 
\begin{equation}\label{equ_mathbbbb}
\mathbf{b} \coloneqq (b_1, \dots, b_m) \in \Z^m
\end{equation}
such that
\begin{equation}\label{equ_boundaryfree2}
\partial \beta + \sum_{j=1}^m b_j \partial \beta_j = 0.
\end{equation}
Under the identification $\pi_2(G/B, N) \simeq \pi_2(G/B) \times \pi_1(N)$, the class
\begin{equation}\label{equ_sphericalalpha}
\alpha \coloneqq \beta + \sum_{j=1}^m b_j \beta_j.
\end{equation}
lies in $\pi_2(G/B)$ and hence defines a spherical class.

We now introduce the boundary lifting map, which will be used to describe the Laurent monomial associated with each lifted holomorphic disk. Let the spherical class $\alpha$ be expressed in the Schubert basis as
\begin{equation}\label{equ_alphaclass}
\alpha = \sum_{i=1}^n a_i [X_{s_i}]. 
\end{equation}
We define the coordinate vector of $\alpha$ with respect to the Schubert basis by
\begin{equation}\label{equ_aaaa}
\mathbf{a} \coloneqq  (a_1, \dots, a_n) \in \Z^n.
\end{equation}

\begin{definition}\label{def_boundaryliftingmap}
Fix the homotopy classes in~\eqref{equ_beta1m}. The \emph{boundary lifting map} 
$$
\partial^\sharp \colon \pi_2(G/B, N) \to \Z^n \times \Z^m
$$
is defined as follows. Given a class $\beta \in \pi_2(G/B, N)$, let $\mathbf{b}$ be the unique vector determined by~\eqref{equ_boundaryfree2} and let
$
\alpha=\beta + \sum_{j=1}^{m}b_j\beta_j
$
be the spherical class defined in~\eqref{equ_sphericalalpha}. Let $\mathbf{a}$ denote the coordinate vector of $\alpha$ with respect to the Schubert basis in~\eqref{equ_aaaa}. We then define
$$
\partial^\sharp (\beta) \coloneqq (-\mathbf{a}, -\mathbf{b}).
$$
\end{definition}

Suppose that the level $\nu$ in~\eqref{equ_choiceofnu} is chosen to be a positive multiple of $(1,\dots,1)$. Then the corresponding K\"{a}hler form $\omega_{\lambda}$ on $G/B$ is monotone. Throughout the remainder of this subsection, we take $\lambda$ to be twice the sum of the fundamental weights. With this choice,
\begin{equation}\label{equ_sumoffundweights}
\nu = (2, \dots, 2)
\end{equation}
and hence $\mathcal L_\lambda \cong K_{G/B}^{-1}$.

Since $G/U$ deformation retracts onto $\mu_T^{-1}(\nu)$ while fixing $L$, there is an isomorphism 
$$
\pi_2(G/U, L) \simeq \pi_2(\mu_T^{-1}(\nu), L).
$$ 
Together with~\eqref{equ_Kahlerforms}, this implies that the symplectic area is preserved under the induced homomorphism
$$
q_* \colon \pi_2(G/U, L) \to \pi_2(G/B,N).
$$ 
Moreover, by Lemma~\ref{lemma_factsfromKIM}, the Maslov index is also preserved. Therefore, we obtain the following.

\begin{lemma}\label{lemma_monotoneLagrangiantori}
Let $L \coloneqq q_\nu^{-1}(N)$. Then $L$ is monotone if and only if $N$ is monotone.
\end{lemma}

We now state the main theorem of this section.

\begin{theorem}\label{theorem_maintheorem2}
Assume that the complex structure $J$ is regular, $N$ is a monotone Lagrangian torus, and $L$ is a Lagrangian torus. Then the disk potential of $L$ in $G/U$ equals
\begin{equation}\label{equ_diskpoWLGU}
W_L (\mathbf{y}) = \sum_{\beta \in \pi_2(G/B, N)} n_\beta \, \mathbf{y}^{\partial^{\sharp}(\beta)}
\end{equation}
where $\partial^{\sharp}(\beta)$ is the image of $\beta$ under the boundary lifting map in Definition~\ref{def_boundaryliftingmap}.
\end{theorem}

To apply the boundary lifting lemma, consider the standard representation $\rho \colon B/U \to \mathrm{GL}(\mathbb{C}^{n+1})$ given by
$$
b * (v_1, \dots, v_{n+1}) = (b_{1,1} v_1, \dots, b_{n+1,n+1}v_{n+1}).
$$
Let $\{ e_1, \dots, e_{n+1} \}$ denote the standard basis of $\mathbb{C}^{n+1}$. Since $B/U$ acts diagonally, the representation decomposes into one dimensional weight spaces
$$
\mathbb{C}^{n+1} = V_1 \oplus \dots \oplus V_{n+1} \quad \mbox{where $V_j = \mathbb{C} e_j$.}
$$
Let $\epsilon_j$ be the character given by $\epsilon_j \bigl( \mathrm{diag}(t_1, \dots, t_{n+1}) \bigr) = t_j$. Then $\rho = \bigoplus_{j=1}^{n+1} \epsilon_j$.

Let $\mathcal{E}^\prime \coloneqq G/U \times_{\rho} \mathbb{C}^{n+1}$ be the vector bundle over $G/B$ associated with $\rho$. For each $j \in [n+1]$, let $\mathcal{L}_j \coloneqq G/U \times_{\epsilon_j} V_j$ be the associated line bundle. Then
\begin{equation}\label{equ_decompositionline}
\mathcal{E}^\prime \simeq \bigoplus_{j=1}^{n+1}  G/U \times_{\epsilon_j} V_j \simeq \bigoplus_{j=1}^{n+1} \mathcal{L}_j.
\end{equation}
For our purposes, we consider the subbundle of rank $n$ given by
$$
\mathcal{E} \coloneqq \bigoplus_{j=1}^{n} \mathcal{L}_j.
$$

For each $j \in [n]$, let $\eta_j \in \pi_2(\mathcal{L}_j,L)$ denote the relative homotopy class represented by a fiberwise disk in $\mathcal{L}_j$ that intersects the zero section transversely and positively at a single point. Via the natural inclusion $\mathcal{L}_j \to \mathcal{E}$, we also regard $\eta_j$ as an element of $\pi_2(\mathcal{E},L)$. Recall that $\pi_1(L) \simeq \pi_1(T) \times \pi_1(N)$. We fix the ordered basis of $\pi_1(L)$ in $G/U$
\begin{equation}\label{equ_sigmabeta}
\sigma_1, \dots, \sigma_n, \partial \beta_1^\sharp, \dots, \partial \beta_m^\sharp.
\end{equation}
where
\begin{equation}\label{equ_identificationofact}
\sigma_j \coloneqq \partial \eta_{j} - \partial \eta_{j+1}.
\end{equation}
Here we set $\partial \eta_{n+1} = 0$ since the line bundle summand $\mathcal{L}_{n+1}$ is omitted from the definition of $\mathcal{E}$. Let $\mathbf{y}$ denote the exponential coordinates corresponding to the basis~\eqref{equ_sigmabeta}. 

We state one more lemma.

\begin{lemma}[See Propositions 1.4.1 and 1.4.3 in \cite{Bri04}] \label{lemma_chernnumbercal}
For each $\lambda \in \mathsf{P}^+$, the associated homogeneous line bundle $\mathcal{L}_\lambda$ is globally generated and
$$
\langle c_1(\mathcal{L}_\lambda) , \alpha \rangle = \sum_{i=1}^n a_i (\lambda_i - \lambda_{i+1}).
$$
In particular, 
$$
\langle c_1(\mathcal{L}_j) , \alpha \rangle = a_j - a_{j-1},
$$
where we set $a_0 \coloneqq 0$.
\end{lemma}

We are ready to prove Theorem~\ref{theorem_maintheorem2}.

\begin{proof}[Proof of Theorem~\ref{theorem_maintheorem2}]
Recall that the flag variety $G/B$ is a smooth Fano variety and simply connected. Because $N$ is monotone, the pair $(G/B, N)$ is positive. Since minimal Chern number of $G/B$ is two,~\eqref{equ_mmst} holds for every homotopy class $\beta$ of Maslov index two. Since $G/B$ is compact,~\eqref{equ_mmcpt} also holds. Note that $\pi_2(G/U) \simeq \{0\}$. By Corollary~\ref{cor_diskposreg}, the disk potential of $L$ is 
$$
W_L (\mathbf{y}) = \sum_{\beta \in \pi_2(M,N)} n_\beta \, \mathbf{y}^{\partial \beta^\sharp}.
$$

Let $\beta\in\pi_2(G/B,N)$ be given and let $\mathbf{b}$ be the unique vector in~\eqref{equ_mathbbbb}. Applying Lemma~\ref{lemma_chernnumbercal} to the line bundle $\mathcal{L}_j$ and $\alpha$ in~\eqref{equ_alphaclass}, the boundary lifting lemma yields that
$$
\partial \beta^\sharp + \sum_{j=1}^m b_j \partial \beta_j^\sharp + \sum_{j=1}^{n} (a_j - a_{j-1}) \partial \eta_j = 0
$$
By~\eqref{equ_identificationofact}, we have
$$
\sum_{j=1}^n a_j \sigma_j = \sum_{j=1}^n a_j (\partial \eta_{j} - \partial \eta_{j+1}) =   \sum_{j=1}^{n} (a_j - a_{j-1}) \partial \eta_j.
$$
Therefore, 
\begin{equation}\label{equ_linearcomofpartialbeta}
\partial \beta^\sharp = - \sum_{j=1}^n a_j \sigma_j - \sum_{j=1}^m b_j \partial \beta_j^\sharp.
\end{equation}
Under the identification $\pi_1(L) \simeq \Z^{n+m}$ determined by~\eqref{equ_sigmabeta}, the class $\partial \beta^\sharp$ coincides with the image of $\beta$ under the boundary lifting map $\partial^\sharp$ in Definition~\ref{def_boundaryliftingmap}. Hence, $W_L$ is equal to~\eqref{equ_diskpoWLGU}.
\end{proof}

\section{GHKK superpotential via disk potential}\label{sec_GHKKstring}

The goal of this section is to recover the Gross--Hacking--Keel--Kontsevich (GHKK) superpotential of $G/U$ by lifting holomorphic disks from $G/B$.

To describe the GHKK superpotential and construct the relevant Lagrangian tori, we begin by reviewing various cluster algebras associated with a reduced expression of the longest element of the Weyl group $W \simeq \mathfrak{S}_{n+1}$ of $G = \mathrm{SL}_{n+1}(\C)$, following \cite{BFZ05, Williams13, GLS11, GY17, FO25}. 

Let $W \simeq \mathfrak{S}_{n+1}$ be its Weyl group. We denote by $w_0$ the longest element of $W$, that is,
$$
w_0 
\coloneqq 
\begin{pmatrix}
1 & 2 & \dots & n+1 \\
n+1 & n & \dots & 1
\end{pmatrix}. 
$$
Fix a reduced expression of $w_0 \colon$
$$
\underline{w}_0 = s_{i_1} s_{i_2} \cdots s_{i_m}
$$
where $m = n(n+1)/2$ is the length of $w_0$. Let
$$
\underline{I} \coloneqq \{ \underline{i} \mid i \in [n] \}
$$
and write $i_{\underline{i}}=i$ for each $\underline{i}\in\underline{I}$.

Set $\underline{I} \coloneqq \{ \underline{i} \mid i \in [n] \}$ and $i_{\underline{i}} \coloneqq i$ for each $\underline{i} \in \underline{I}$.
\begin{itemize}
\item For each $\underline{i} \in \underline{I} $, set $\underline{i}^+ \coloneqq \min \bigl(  \{ m+1 \} \cup \{ 1 \leq j \leq m \mid i_j = i \} \bigr)$ and
\item For each $k \in [m]$, set 
$
k^+ \coloneqq \min \bigl(  \{ m+1 \} \cup \{  k+1 \leq j \leq m \mid i_j = i_k \} \bigr)$. 
\end{itemize}
We then define the index sets
\begin{itemize}
\item $J \coloneqq [m], \, J_\mathrm{fz} \coloneqq \{ j \in J \mid j^+ = m+1 \}, \mbox{ and } J_\mathrm{uf} \coloneqq J \setminus J_\mathrm{fz}$ and
\item $\underline{J} \coloneqq \underline{I} \cup J, \underline{J}_\mathrm{fz} \coloneqq \underline{I} \cup J_\mathrm{fz}, \mbox{ and }  \underline{J}_\mathrm{uf} \coloneqq \underline{J} \setminus \underline{J}_\mathrm{fz} = J_\mathrm{uf}.$
\end{itemize}

We define two cluster algebras $\mathcal{A}$ and $\underline{\mathcal{A}}$. Using the Cartan matrix $\mathsf{A}$ of $\mathfrak{g}$, for each ${r \in \underline{J}_{\mathrm{uf}},\, s \in \underline{J}}$, define the integer $\varepsilon_{r,s}$ by
\begin{equation}\label{equ_GLSseed}
\varepsilon_{r,s} =
\begin{cases}
-1 & \text{if } r = s^+, \\
1 & \text{if } r^+ = s, \\
-a_{i_s,i_r} & \text{if } s < r < s^+ < r^+, \\
a_{i_s,i_r} & \text{if } r < s < r^+ < s^+, \\
0 & \text{otherwise.}
\end{cases}
\end{equation}

Take the initial seed 
$$
\mathsf{s}_0 \coloneqq (A_{\mathsf{s}_0}, \varepsilon_{\mathsf{s}_0}) \quad \mbox{ and } 
\quad \underline{\mathsf{s}}_0 \coloneqq (A_{\underline{\mathsf{s}}_0}, \varepsilon_{\underline{\mathsf{s}}_0})
$$ 
where $A_{\mathsf{s}_0} \coloneqq \{ A_{j} \mid j \in J \}$ and $A_{\underline{\mathsf{s}}_0} \coloneqq \{ A_{j} \mid j \in \underline{J} \}$ are algebraically independent variables and the extended exchange matrices $\varepsilon_{\mathsf{s}_0} \coloneqq (\varepsilon_{r,s})_{r \in J_{\mathrm{uf}},\, s \in J}$ and $\varepsilon_{\underline{\mathsf{s}}_0} \coloneqq (\varepsilon_{r,s})_{r \in \underline{J}_{\mathrm{uf}},\, s \in \underline{J}}$ given by~\eqref{equ_GLSseed}. For simplicity, we write $\mathsf{s}(\underline{w}_0)$ or simply $\underline{w}_0$, for the seed associated with the reduced expression $\underline{w}_0$.

\begin{definition}
The \emph{cluster algebra} $\mathcal{A}$ (resp. $\underline{\mathcal{A}}$) is the $\C$-subalgebra of $\C( A_{j} \mid j \in J )$ (resp. $\C( A_{j} \mid j \in \underline{J})$) generated by all cluster variables in seeds mutation equivalent to the initial seed $\mathsf{s}_0$ (resp. $\underline{\mathsf{s}}_0$).  

The \emph{upper cluster algebra}, denoted by $\mathcal{A}^\mathrm{up}$ (resp. $\underline{\mathcal{A}}^\mathrm{up}$), is defined by 
$$
\mathcal{A}^\mathrm{up} = \bigcap_{\mathsf{s}} \C[A^\pm_\mathsf{s}] \left(\mbox{resp.} \,\, \underline{\mathcal{A}}^\mathrm{up} = \bigcap_{\underline{\mathsf{s}}} \C[A^\pm_{\underline{\mathsf{s}}}]  \right).
$$ 
\end{definition}

The double Bruhat cell is defined by $G^{e, w_0} \coloneqq B^{-} \cap Bw_0B$ and the unipotent cell $U^-_{w_0} \coloneqq U^{-} \cap BwB$. We have open embeddings
$$
G^{e, w_0} \to G / U \mbox{ and } U^-_{w_0} \to G / B.
$$
By \cite{BFZ05, FO25}, the upper cluster algebra $\mathcal{A}^\mathrm{up}$ is isomorphic to the coordinate ring $\C[U^-_{w_0}]$ of the unipotent cell and the upper cluster algebra $\underline{\mathcal{A}}^\mathrm{up}$ is isomorphic to   the coordinate ring $\C[G^{e,w_0}]$ of the double Bruhat cell.

In \cite{Mag15}, Magee proved that the full FG conjecture holds for $G/U$. The full Fock--Goncharov conjecture \cite[Definition 0.6]{GHKK18} yields that the $\vartheta$-basis of the coordinate ring $\C[\underline{\mathcal{A}}]$ is parametrized by the tropical points of the dual cluster variety $\mathcal{X}$. Let ${\underline{\mathcal{A}}^c}$ be a partial compactification $\underline{\mathcal{A}}$ by allowing the frozen variables to take the value $0$. Every regular function on the partial compactification ${\underline{\mathcal{A}}^c}$ has a non-negative order along the divisor $V(A_j)$ for each $j \in J_\mathrm{fz}$. Let $\mathcal{X}$ be the cluster $\mathcal{X}$-variety dual to the cluster variety $\underline{\mathcal{A}}$. The superpotential is given by 
\begin{equation}\label{equ_Wtheta}
W \coloneqq \sum_{j \in \underline{J}_\mathrm{fz}} \vartheta_j
\end{equation}
where $\vartheta_j$ is the theta function corresponding to the frozen variable $A_j$. The tropical points parametrizing the regular functions on ${\underline{\mathcal{A}}^c}$ are the integral points lying in the intersection of the half spaces determined by the tropicalization of $W$. 

To obtain the explicit expressions of $W|_{\mathcal{X}_{\underline{w}_0}}$ in terms of cluster coordinates, we recall wiring diagrams.

\begin{definition}[{\cite[Section 5.1]{GP00}}]\label{definition_rigorous_path}
Let $\underline{w}_0 = s_{i_1}s_{i_2} \dots s_{i_m}$ be a reduced expression of the longest element $w_0$ of the Weyl group $W \simeq \mathfrak{S}_{n+1}$ where $m \coloneqq n(n+1)/2$. 

\begin{itemize}
\item The \emph{wiring diagram} ${G}_{\underline{w}_0}$ is an arrangement of $(n+1)$-vertical piecewise lines such that 
\begin{enumerate}
\item each pair of wires must intersect exactly once and 
\item for each $j \in [m]$, the $j$-th crossing of wires from top occurs in the $i_j$-th column from the left of ${G}_{\underline{w}_0}$ .
\end{enumerate}
\item The wiring diagram divides the plane into \emph{chambers}. It has $n(n-1)/2$ bounded chambers together with $n$ unbounded chambers at the bottom and $n$ unbounded chambers at the top. Let $L_k$ denote the $k$-th vertex on the bottom boundary from the left and let $\ell_k$ denote the $k$-th wire from the top.
\item For each $k \in [n+1]$, let $G_{\underline{w}_0}(k)$ denote the oriented wiring diagram obtained from $G_{\underline{w}_0}$ by orienting$\colon$
\begin{enumerate}
\item the first $k$ wires $\ell_1, \ldots, \ell_k$ upward, and 
\item the remaining wires $\ell_{k+1}, \ldots, \ell_{n+1}$ downward. 
\end{enumerate}
\item A \emph{rigorous path} is an oriented path on $G_{\underline{w}_0}(k)$ for some $k \in [n+1]$ satisfying
\begin{enumerate}
\item it starts at $L_k$ and ends at $L_{k+1}$, 
\item it respects the orientation of $G_{\underline{w}_0}(k)$,
\item it passes through each node at most once, and
\item it does \emph{not} include \emph{forbidden fragments} shown in Figure~\ref{figure_avoiding}.
\end{enumerate}
We denote by $\mathcal{P}_{\underline{w}_0}$ the set of all rigorous paths. For each $k \in [n+1]$, we denote by $\mathcal{P}_{\underline{w}_0}(k)$ the set of rigorous paths in $G_{\underline{w}_0}(k)$.
		
\begin{figure}[h]
\begin{center}
\begin{tikzpicture}
\tikzset{red line/.style = {line width=0.5ex, red}}
\draw[->] (0,1)--(1,0);
\draw[red line, ->] (1,1)--(0,0);
\begin{scope}[xshift = 2.5cm]
\draw[->] (0,0)--(1,1);
\draw[->, red line] (1,0)--(0,1);
\end{scope}					
\end{tikzpicture}	
\end{center}
\caption{\label{figure_avoiding} Forbidden fragments.}
\end{figure}
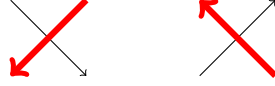
\item A \emph{descending array of chambers} is a sequence of chambers contained in a single column, beginning with a top chamber and proceeding downward through the successive bounded chambers. The sequence must terminate at either a top chamber or a bounded chamber in that column. We denote by $\mathcal{Q}_{\underline{w}_0}$ the set of descending arrays of chambers and by $\mathcal{Q}_{\underline{w}_0}(i)$ the subset consisting of descending arrays in the $i$-th column from the left.
\item An \emph{ascending array of chambers} is a sequence of chambers contained in a single column, beginning with a bottom chamber and proceeding upward through the successive bounded chambers. The sequence must terminate at either a bottom chamber or a bounded chamber in that column. We denote by $\mathcal{R}_{\underline{w}_0}$ the set of ascending arrays of chambers and by $\mathcal{R}_{\underline{w}_0}(i)$ the subset consisting of ascending arrays of chambers in the $i$-th column from the left.
\item For each chamber indexed by $j \in \underline{J}$, we associate a \emph{chamber variable} $u_j$. More precisely,
\begin{enumerate}
\item for each $\underline{j} \in \underline{I}$, the variable $u_{\underline{j}}$ is assigned to the top unbounded chamber between the wires $\ell_j$ and $\ell_{j+1}$
\item for each $j \in J$, the variable $u_j$ is assigned to the chamber immediately below the crossing corresponding to $s_{i_j}$.
\end{enumerate}
\end{itemize}
\end{definition}

\begin{example}\label{example_twored}
Let $G = \mathrm{SL}_{4}(\C)$. As examples, consider two reduced expressions 
$$
\underline{w}_0 = s_1 s_2 s_1 s_3 s_2 s_1 \, \mbox{ and } \, \underline{w}_{0^\prime} = s_2 s_1 s_3 s_2 s_1 s_3
$$ 
of the longest element $w_0$. The wiring diagrams $G_{\underline{w}_0}$ and $G_{\underline{w}_{0^\prime}}$ are depicted in Figure~\ref{figure_wiringoriented} together with the oriented wiring diagrams $G_{\underline{w}_0}(1)$ and $G_{\underline{w}_{0^\prime}}(2)$.
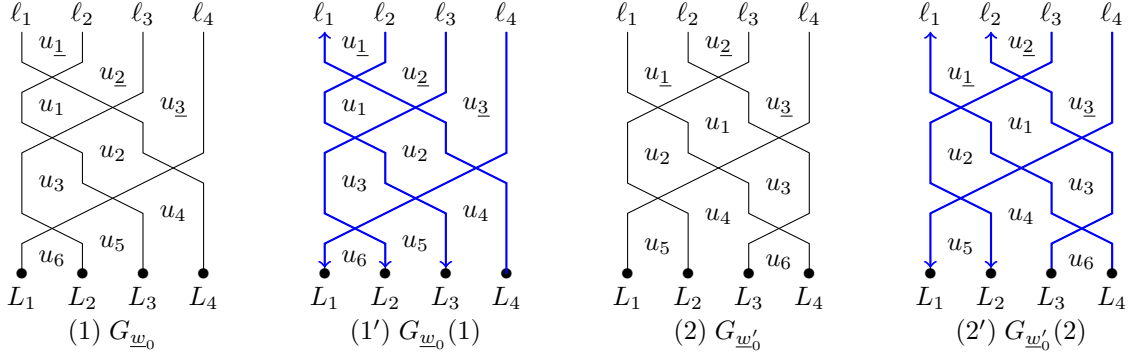
\begin{figure}[h]
\centering
\begin{tikzpicture}[scale=0.8, every node/.style={font=\small}]

\begin{scope}
\node at (0, -0.4) {$L_1$};
\node at (1, -0.4) {$L_2$};
\node at (2, -0.4) {$L_3$};
\node at (3, -0.4) {$L_4$};

\node at (0, 0)  {$\bullet$};
\node at (1, 0)  {$\bullet$};
\node at (2, 0)  {$\bullet$};
\node at (3, 0)  {$\bullet$};

\node at (0, 4+0.3) {$\ell_1$};
\node at (1, 4+0.3) {$\ell_2$};
\node at (2, 4+0.3) {$\ell_3$};
\node at (3, 4+0.3) {$\ell_4$};

\draw (3,0)  -- (3,1.5) -- (2,2) -- (2,2.5) -- (0,3.5) -- (0,4);
\draw (2,0) -- (2,1) -- (1,1.5) -- (1,2) -- (0,2.5) -- (0,3) -- (1,3.5) -- (1,4);
\draw (1,0) -- (1,0.5) -- (0,2.5-1.5) -- (0,2) -- (2,3) -- (2,4);
\draw (0,0) -- (0,0.5) -- (3,3.5-1.5) -- (3,4);

\node at (0.5,3.75) {$u_{\underline{1}}$};
\node at (1.5,3.25) {$u_{\underline{2}}$};
\node at (2.5,2.75) {$u_{\underline{3}}$};
\node at (0.5,2.75) {$u_1$};
\node at (0.5,1.5) {$u_3$};
\node at (0.5,0.25) {$u_6$};
\node at (1.5,2.0) {$u_2$};
\node at (1.5,0.5) {$u_5$};
\node at (2.5,1) {$u_4$};

\node at (1.5,-1) {$(1)\; G_{\underline{w}_0}$};
\end{scope}

\begin{scope}[xshift=5cm]
\node at (0.5,3.75) {$u_{\underline{1}}$};
\node at (1.5,3.25) {$u_{\underline{2}}$};
\node at (2.5,2.75) {$u_{\underline{3}}$};

\node at (0, 0)  {$\bullet$};
\node at (1, 0)  {$\bullet$};
\node at (2, 0)  {$\bullet$};
\node at (3, 0)  {$\bullet$};

\node at (0, 4+0.3) {$\ell_1$};
\node at (1, 4+0.3) {$\ell_2$};
\node at (2, 4+0.3) {$\ell_3$};
\node at (3, 4+0.3) {$\ell_4$};

\draw (3,0)  -- (3,1.5) -- (2,2) -- (2,2.5) -- (0,3.5) -- (0,4);
\draw (2,0) -- (2,1) -- (1,1.5) -- (1,2) -- (0,2.5) -- (0,3) -- (1,3.5) -- (1,4);
\draw (1,0) -- (1,0.5) -- (0,2.5-1.5) -- (0,2) -- (2,3) -- (2,4);
\draw (0,0) -- (0,0.5) -- (3,3.5-1.5) -- (3,4);

\draw[blue, thick, ->] (3,0)  -- (3,1.5) -- (2,2) -- (2,2.5) -- (0,3.5) -- (0,4);
\draw[blue, thick, <-] (2,0.1) -- (2,1) -- (1,1.5) -- (1,2) -- (0,2.5) -- (0,3) -- (1,3.5) -- (1,4);
\draw[blue, thick, <-] (1,0.1) -- (1,0.5) -- (0,2.5-1.5) -- (0,2) -- (2,3) -- (2,4);
\draw[blue, thick, <-] (0,0.1) -- (0,0.5) -- (3,3.5-1.5) -- (3,4);

\node at (0, -0.4) {$L_1$};
\node at (1, -0.4) {$L_2$};
\node at (2, -0.4) {$L_3$};
\node at (3, -0.4) {$L_4$};

\node at (0.5,2.75) {$u_1$};
\node at (0.5,1.5) {$u_3$};
\node at (0.5,0.25) {$u_6$};
\node at (1.5,2.0) {$u_2$};
\node at (1.5,0.5) {$u_5$};
\node at (2.5,1) {$u_4$};

\node at (1.5,-1) {$(1^\prime)\; G_{\underline{w}_0}(1)$};
\end{scope}

\begin{scope}[xshift=10cm]
\node at (0.5,3.25) {$u_{\underline{1}}$};
\node at (1.5,3.75) {$u_{\underline{2}}$};
\node at (2.5,2.75) {$u_{\underline{3}}$};

\node at (0, -0.4) {$L_1$};
\node at (1, -0.4) {$L_2$};
\node at (2, -0.4) {$L_3$};
\node at (3, -0.4) {$L_4$};

\node at (0, 0)  {$\bullet$};
\node at (1, 0)  {$\bullet$};
\node at (2, 0)  {$\bullet$};
\node at (3, 0)  {$\bullet$};
\node at (0, 4+0.3) {$\ell_1$};
\node at (1, 4+0.3) {$\ell_2$};
\node at (2, 4+0.3) {$\ell_3$};
\node at (3, 4+0.3) {$\ell_4$};

\draw (0,0) -- (0,1) -- (1,1.5) -- (3,2.5) -- (3,4) ;
\draw (1,0) -- (1,1) -- (0,1.5) -- (0,2.5) -- (2,3.5) -- (2,4);
\draw (2,0) -- (2,0.5) -- (3,1) -- (3,1.5) -- (3,2) -- (2,2.5) -- (2,3) -- (1,3.5) -- (1,4);
\draw (3,0) -- (3,0.5) -- (2,1) -- (2,1.5) -- (1,2) -- (1,2.5) -- (0,3) -- (0,4);

\node at (2.5,0+0.2) {$u_6$};
\node at (0.5,1.75+0.2) {$u_2$};
\node at (0.5,0.25+0.2) {$u_5$};
\node at (1.5,2.25+0.2) {$u_1$};
\node at (1.5,0.75+0.2) {$u_4$};
\node at (2.5,1.25+0.2) {$u_3$};

\node at (1.5,-1) {$(2)\; G_{\underline{w}^\prime_0}$};
\end{scope}

\begin{scope}[xshift=15cm]
\node at (0.5,3.25) {$u_{\underline{1}}$};
\node at (1.5,3.75) {$u_{\underline{2}}$};
\node at (2.5,2.75) {$u_{\underline{3}}$};

\node at (0, 4+0.3) {$\ell_1$};
\node at (1, 4+0.3) {$\ell_2$};
\node at (2, 4+0.3) {$\ell_3$};
\node at (3, 4+0.3) {$\ell_4$};

\draw (0,0) -- (0,1) -- (1,1.5) -- (3,2.5) -- (3,4) ;
\draw (1,0) -- (1,1) -- (0,1.5) -- (0,2.5) -- (2,3.5) -- (2,4);
\draw (2,0) -- (2,0.5) -- (3,1) -- (3,1.5) -- (3,2) -- (2,2.5) -- (2,3) -- (1,3.5) -- (1,4);
\draw (3,0) -- (3,0.5) -- (2,1) -- (2,1.5) -- (1,2) -- (1,2.5) -- (0,3) -- (0,4);

\draw[blue, thick, ->] (3,0) -- (3,0.5) -- (2,1) -- (2,1.5) -- (1,2) -- (1,2.5) -- (0,3) -- (0,4);
\draw[blue, thick, ->] (2,0) -- (2,0.5) -- (3,1) -- (3,1.5) -- (3,2) -- (2,2.5) -- (2,3) -- (1,3.5) -- (1,4);
\draw[blue, thick, <-] (1,0.1) -- (1,1) -- (0,1.5) -- (0,2.5) -- (2,3.5) -- (2,4);
\draw[blue, thick, <-] (0,0.1) -- (0,1) -- (1,1.5) -- (3,2.5) -- (3,4) ;

\node at (0, -0.4) {$L_1$};
\node at (1, -0.4) {$L_2$};
\node at (2, -0.4) {$L_3$};
\node at (3, -0.4) {$L_4$};

\node at (0, 0)  {$\bullet$};
\node at (1, 0)  {$\bullet$};
\node at (2, 0)  {$\bullet$};
\node at (3, 0)  {$\bullet$};

\node at (2.5,0+0.2) {$u_6$};
\node at (0.5,1.75+0.2) {$u_2$};
\node at (0.5,0.25+0.2) {$u_5$};
\node at (1.5,2.25+0.2) {$u_1$};
\node at (1.5,0.75+0.2) {$u_4$};
\node at (2.5,1.25+0.2) {$u_3$};

\node at (1.5,-1) {$(2^\prime)\; G_{\underline{w}^\prime_0}(2)$};
\end{scope}
\end{tikzpicture}
\caption{\label{figure_wiringoriented} Wiring diagrams and oriented wiring diagrams.}
\end{figure}
\end{example}

Magee and Bossinger--Fourier gave an explicit expression for the GHKK superpotential restricted to the $\mathcal{X}$-cluster chart associated with $\mathsf{s}(\underline{w}_0) = \underline{w}_0$. We briefly recall their construction. 

Consider the Euclidean space $\{(\underline{\mathbf{u}}, {\mathbf{u}})\} \simeq \R^n \times \R^m$.
\begin{enumerate}
\item For each rigorous path $\mathsf{p} \in \mathcal{P}_{\underline{w}_0}$, let $h_{\mathsf{p}}(\mathbf{u})$ denote the sum of the chamber variables corresponding to the chambers enclosed by $\mathsf{p}$. Let $\mathbf{v}_\mathsf{p}$ be the primitive inward  normal vector to the half space
\begin{equation}\label{equ_stringincone}
h_{\mathsf{p}}(\mathbf{u}) \geq 0.
\end{equation}
\item For each descending array of chambers $\mathsf{q}\in\mathcal{Q}_{\underline{w}_0}$, let $h_{\mathsf{q}}(\underline{\mathbf{u}}, {\mathbf{u}})$ denote the sum of the chamber variables corresponding to the chambers in the array. Let $\mathbf{v}_\mathsf{q}$ be the primitive inward normal vector to the half space
$$
h_{\mathsf{q}}(\underline{\mathbf{u}}, {\mathbf{u}}) \geq 0.
$$
\end{enumerate}

\begin{theorem}[Corollary 24 in \cite{Mag15} and Theorem 6 in \cite{BF19}]
The GHKK superpotential for $G/U$ restricted to the cluster chart $\mathcal{X}_{\mathsf{s}(\underline{w}_0)} \simeq (\C^*)^{n+m}$ is equal to
\begin{equation}\label{equ_GHKKssupter}
W_{\underline{w}_0}(\mathbf{x}) = \sum_{\mathsf{p} \in \mathcal{P}_{\underline{w}_0}} \mathbf{x}^{-\mathbf{v}_\mathsf{p}} +  \sum_{\mathsf{q} \in \mathcal{Q}_{\underline{w}_0}} \mathbf{x}^{-\mathbf{v}_\mathsf{q}} 
\end{equation}
\end{theorem}

\begin{example}\label{example_GHKKsuperpo}
Let $\underline{w}_0$ and $\underline{w}_{0^\prime}$ be the two reduced expressions of the longest element $w_0$ of $\mathfrak{S}_4$ in Example~\ref{example_twored}. Consider the rigorous paths $\mathsf{p}\in\mathcal{P}_{\underline{w}_0}$, $\mathsf{p}^\prime \in \mathcal{P}_{\underline{w}_{0^\prime}}$ and the descending arrays $\mathsf{q} \in \mathcal{Q}_{\underline{w}_0}(1)$, $\mathsf{q}^\prime \in \mathcal{Q}_{\underline{w}_{0^\prime}}(2)$ shown in Figure~\ref{figure_arraypath} as examples. The corresponding defining linear functions are
$$
h_{\mathsf{p}}(\mathbf{u}) =u_2 + u_4, \,\, h_{\mathsf{p}^\prime}(\mathbf{u}) = u_2 + u_3 + u_4, \,\, h_{\mathsf{q}}(\underline{\mathbf{u}}, \mathbf{u}) = u_{\underline{1}} + u_1 + u_3,   \,\, h_{\mathsf{q}^\prime}(\underline{\mathbf{u}}, \mathbf{u}) = u_{\underline{2}} + u_1. 
$$
Then the GHKK superpotentials are
{\small
\begin{align*}
&W_{\underline{w}_0}(\mathbf{x}) = \frac{1}{x_1x_2x_4} + \frac{1}{x_2x_4} + \frac{1}{x_4} + \frac{1}{x_3x_5} + \frac{1}{x_5} + \frac{1}{x_6} + \frac{1}{x_{\underline{1}}} + \frac{1}{x_{\underline{1}}x_1} + \frac{1}{x_{\underline{1}}x_1x_3} +  \frac{1}{x_{\underline{2}}} + \frac{1}{x_{\underline{2}}x_2} +  \frac{1}{x_{\underline{3}}}
\\ 
&W_{\underline{w}_{0^\prime}}(\mathbf{x}) = \frac{1}{x_1x_2x_3x_4}  + \frac{1}{x_2x_3x_4} + \frac{1}{x_2x_4} + \frac{1}{x_3x_4} + \frac{1}{x_4} + \frac{1}{x_5} + \frac{1}{x_6} + \frac{1}{x_{\underline{1}}} + \frac{1}{x_{\underline{1}}x_2} +  \frac{1}{x_{\underline{2}}} +  \frac{1}{x_{\underline{2}}x_1} + \frac{1}{x_{\underline{3}}} +  \frac{1}{x_{\underline{3}}x_3}.
\end{align*}
}
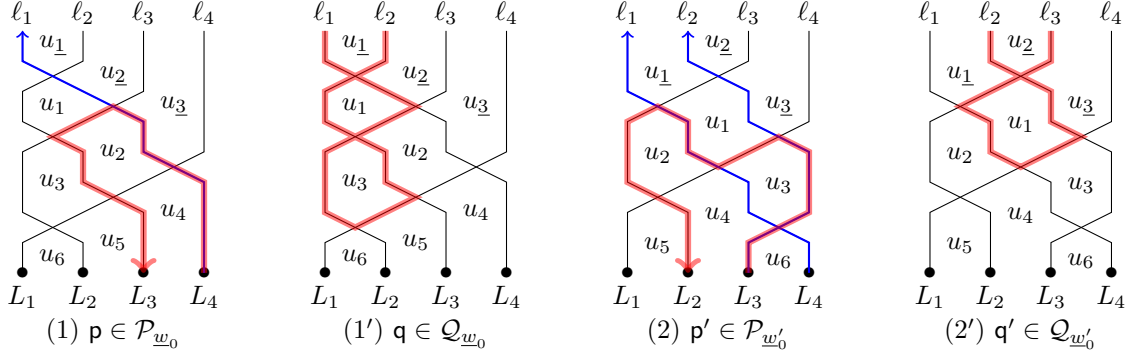
\begin{figure}[h]
\centering
\begin{tikzpicture}[scale=0.8, every node/.style={font=\small}]

\begin{scope}
\node at (0, -0.4) {$L_1$};
\node at (1, -0.4) {$L_2$};
\node at (2, -0.4) {$L_3$};
\node at (3, -0.4) {$L_4$};

\node at (0, 0)  {$\bullet$};
\node at (1, 0)  {$\bullet$};
\node at (2, 0)  {$\bullet$};
\node at (3, 0)  {$\bullet$};

\node at (0, 4+0.3) {$\ell_1$};
\node at (1, 4+0.3) {$\ell_2$};
\node at (2, 4+0.3) {$\ell_3$};
\node at (3, 4+0.3) {$\ell_4$};

\draw (3,0)  -- (3,1.5) -- (2,2) -- (2,2.5) -- (0,3.5) -- (0,4);
\draw (2,0) -- (2,1) -- (1,1.5) -- (1,2) -- (0,2.5) -- (0,3) -- (1,3.5) -- (1,4);
\draw (1,0) -- (1,0.5) -- (0,2.5-1.5) -- (0,2) -- (2,3) -- (2,4);
\draw (0,0) -- (0,0.5) -- (3,3.5-1.5) -- (3,4);
\draw[blue, thick, ->] (3,0)  -- (3,1.5) -- (2,2) -- (2,2.5) -- (0,3.5) -- (0,4);
\draw[red, semitransparent, line width=0.5ex, ->] (3,0)--(3,1.5)--(2,2)--(2,2.5)--(1.5,2.75)--(0.5,2.25)--(1,2)--(1,1.5)--(2,1)--(2,0);

\node at (0.5,3.75) {$u_{\underline{1}}$};
\node at (1.5,3.25) {$u_{\underline{2}}$};
\node at (2.5,2.75) {$u_{\underline{3}}$};
\node at (0.5,2.75) {$u_1$};
\node at (0.5,1.5) {$u_3$};
\node at (0.5,0.25) {$u_6$};
\node at (1.5,2.0) {$u_2$};
\node at (1.5,0.5) {$u_5$};
\node at (2.5,1) {$u_4$};

\node at (1.5,-1) {$(1)$ $\mathsf{p} \in \mathcal{P}_{\underline{w}_0}$};
\end{scope}

\begin{scope}[xshift=5cm]
\node at (0, 0)  {$\bullet$};
\node at (1, 0)  {$\bullet$};
\node at (2, 0)  {$\bullet$};
\node at (3, 0)  {$\bullet$};

\node at (0, 4+0.3) {$\ell_1$};
\node at (1, 4+0.3) {$\ell_2$};
\node at (2, 4+0.3) {$\ell_3$};
\node at (3, 4+0.3) {$\ell_4$};

\draw (3,0)  -- (3,1.5) -- (2,2) -- (2,2.5) -- (0,3.5) -- (0,4);
\draw (2,0) -- (2,1) -- (1,1.5) -- (1,2) -- (0,2.5) -- (0,3) -- (1,3.5) -- (1,4);
\draw (1,0) -- (1,0.5) -- (0,2.5-1.5) -- (0,2) -- (2,3) -- (2,4);
\draw (0,0) -- (0,0.5) -- (3,3.5-1.5) -- (3,4);

\draw[red, semitransparent, line width=0.5ex, -] (0,4)--(0,3.5)--(1.5,2.75)--(0.5,2.25) -- (1,2) -- (1,1.5) -- (1.5,1.25) -- (0.5,0.75) -- (0,1) -- (0, 2) -- (0.5, 2.25) 
--(0,2.5) --(0,3) -- (1,3.5) -- (1,4);

\node at (0, -0.4) {$L_1$};
\node at (1, -0.4) {$L_2$};
\node at (2, -0.4) {$L_3$};
\node at (3, -0.4) {$L_4$};

\node at (0.5,3.75) {$u_{\underline{1}}$};
\node at (1.5,3.25) {$u_{\underline{2}}$};
\node at (2.5,2.75) {$u_{\underline{3}}$};

\node at (0.5,2.75) {$u_1$};
\node at (0.5,1.5) {$u_3$};
\node at (0.5,0.25) {$u_6$};
\node at (1.5,2.0) {$u_2$};
\node at (1.5,0.5) {$u_5$};
\node at (2.5,1) {$u_4$};

\node at (1.5,-1) {$(1^\prime)$ $\mathsf{q} \in \mathcal{Q}_{\underline{w}_0}$};
\end{scope}

\begin{scope}[xshift=10cm]
\node at (0.5,3.25) {$u_{\underline{1}}$};
\node at (1.5,3.75) {$u_{\underline{2}}$};
\node at (2.5,2.75) {$u_{\underline{3}}$};

\node at (0, -0.4) {$L_1$};
\node at (1, -0.4) {$L_2$};
\node at (2, -0.4) {$L_3$};
\node at (3, -0.4) {$L_4$};

\node at (0, 0)  {$\bullet$};
\node at (1, 0)  {$\bullet$};
\node at (2, 0)  {$\bullet$};
\node at (3, 0)  {$\bullet$};
\node at (0, 4+0.3) {$\ell_1$};
\node at (1, 4+0.3) {$\ell_2$};
\node at (2, 4+0.3) {$\ell_3$};
\node at (3, 4+0.3) {$\ell_4$};

\draw (0,0) -- (0,1) -- (1,1.5) -- (3,2.5) -- (3,4) ;
\draw (1,0) -- (1,1) -- (0,1.5) -- (0,2.5) -- (2,3.5) -- (2,4);
\draw (2,0) -- (2,0.5) -- (3,1) -- (3,1.5) -- (3,2) -- (2,2.5) -- (2,3) -- (1,3.5) -- (1,4);
\draw (3,0) -- (3,0.5) -- (2,1) -- (2,1.5) -- (1,2) -- (1,2.5) -- (0,3) -- (0,4);

\draw[blue, thick, ->] (3,0) -- (3,0.5) -- (2,1) -- (2,1.5) -- (1,2) -- (1,2.5) -- (0,3) -- (0,4);
\draw[blue, thick, ->] (2,0) -- (2,0.5) -- (3,1) -- (3,1.5) -- (3,2) -- (2,2.5) -- (2,3) -- (1,3.5) -- (1,4);
\draw[red, semitransparent, line width=0.5ex, ->] (2,0) -- (2,0.5) -- (3,1) -- (3,1.5) -- (3,2) -- (2.5,2.25) -- (1.5,1.75) -- (1,2) -- (1,2.5) -- (0.5,2.75) -- (0,2.5) -- (0,1.5) -- (1,1) -- (1,0);

\node at (2.5,0+0.2) {$u_6$};
\node at (0.5,1.75+0.2) {$u_2$};
\node at (0.5,0.25+0.2) {$u_5$};
\node at (1.5,2.25+0.2) {$u_1$};
\node at (1.5,0.75+0.2) {$u_4$};
\node at (2.5,1.25+0.2) {$u_3$};

\node at (1.5,-1) {$(2)$ $\mathsf{p}^\prime \in \mathcal{P}_{\underline{w}^\prime_0}$};

\end{scope}

\begin{scope}[xshift=15cm]
\node at (0, 4+0.3) {$\ell_1$};
\node at (1, 4+0.3) {$\ell_2$};
\node at (2, 4+0.3) {$\ell_3$};
\node at (3, 4+0.3) {$\ell_4$};

\draw (0,0) -- (0,1) -- (1,1.5) -- (3,2.5) -- (3,4) ;
\draw (1,0) -- (1,1) -- (0,1.5) -- (0,2.5) -- (2,3.5) -- (2,4);
\draw (2,0) -- (2,0.5) -- (3,1) -- (3,1.5) -- (3,2) -- (2,2.5) -- (2,3) -- (1,3.5) -- (1,4);
\draw (3,0) -- (3,0.5) -- (2,1) -- (2,1.5) -- (1,2) -- (1,2.5) -- (0,3) -- (0,4);

\draw[red, semitransparent, line width=0.5ex, -] (1,4) -- (1,3.5) -- (2,3) -- (2,2.5) -- (2.5,2.25) -- (1.5,1.75) -- (1,2) -- (1,2.5) -- (0.5,2.75) -- (2,3.5) -- (2,4);

\node at (0, -0.4) {$L_1$};
\node at (1, -0.4) {$L_2$};
\node at (2, -0.4) {$L_3$};
\node at (3, -0.4) {$L_4$};

\node at (0, 0)  {$\bullet$};
\node at (1, 0)  {$\bullet$};
\node at (2, 0)  {$\bullet$};
\node at (3, 0)  {$\bullet$};

\node at (0.5,3.25) {$u_{\underline{1}}$};
\node at (1.5,3.75) {$u_{\underline{2}}$};
\node at (2.5,2.75) {$u_{\underline{3}}$};

\node at (2.5,0+0.2) {$u_6$};
\node at (0.5,1.75+0.2) {$u_2$};
\node at (0.5,0.25+0.2) {$u_5$};
\node at (1.5,2.25+0.2) {$u_1$};
\node at (1.5,0.75+0.2) {$u_4$};
\node at (2.5,1.25+0.2) {$u_3$};
\node at (1.5,-1) {$(2^\prime)$ $\mathsf{q}^\prime \in \mathcal{Q}_{\underline{w}^\prime_0}$};
\end{scope}
\end{tikzpicture}
\caption{\label{figure_arraypath} Rigorous paths and descending arrways}
\end{figure}
\end{example}

Next, we show that these superpotentials are indeed the disk potentials, as conjectured in \cite{GHKK18}, by applying the framework developed in Section~\ref{sec_diskpotentialbasicaffinespaces}. For this purpose, we construct a family of Lagrangian tori whose mirror superpotentials coincide with $W$ above. 

Take $\lambda$ to be twice the sum of fundamental weights so that the corresponding homogeneous line bundle $\mathcal{L}_\lambda$ on $G/B$ is the anticanonical line bundle. To each seed $\mathsf{s}$ of the cluster algebra $\mathcal{A}$, one associates a Newton--Okounkov body $\Delta^\lambda_{\mathsf{s}}$ of $G/B$ with the choice of the line bundle $\mathcal{L}_\lambda$, see  \cite{FO25}. Using the resulting toric degenerations and the completely integrable system constructed in \cite{And13, HK15}, a monotone Lagrangian torus associated to $\mathsf{s}$ was constructed in \cite{CKKP25}. We denote this family by
$$
\{ N_\mathsf{s} \mid \mathsf{s} \mbox{ is a seed for $\mathcal{A}$}\}.
$$

Let $q_\nu$ be the principal $T$-bundle defined in~\eqref{equ_quotientmapatlevelset}. Define
$$
L_\mathsf{s} \coloneqq q^{-1}_\nu (N_\mathsf{s}).
$$

\begin{proposition}\label{proposition_toricdegtorus}
Each $L_\mathsf{s}$ is a monotone Lagrangian torus in $G/U$.
\end{proposition}

\begin{proof}
To prove that $L_\mathsf{s}$ is a torus, we make use of the toric degeneration
$$
\pi \colon \mathcal{X} \to \C
$$ 
whose central fiber $\mathcal{X}_0 \coloneqq \pi^{-1}(0)$ is the toric variety associated with the Newton--Okounkov polytope $\Delta^\lambda_{\mathsf{s}}$, see \cite{FO25}. Let 
$$
\phi_t \colon \mathcal{X}_1 \to \mathcal{X}_{1-t}
$$ 
be the globally continuous map constructed in \cite{HK15}. The restriction of $\phi_t$ to $N_{\mathsf{s}}$ is smooth and $\phi_t$ is a symplectomorphism on an open dense subset containing $N_{\mathsf{s}}$. Hence, for each $t \in [0,1]$, the image $\phi_t(N_{\mathsf{s}})$ is a Lagrangian torus in $\mathcal{X}_{1-t}$. Moreover, the maps $\phi_t$ define a smooth isotopy
$$
\psi \colon [0,1] \times {N}_{\mathsf{s}} \to \bigcup_{t \in [0,1]} \phi_t(N_\mathsf{s}) (\subseteq \mathcal{X}), \quad (t,x) \mapsto \phi_t(x).
$$

For each $t \in [0,1]$, let
$$
q_t \colon L_t \longrightarrow \phi_t(N_{\mathsf{s}})
$$
be the principal $T$-bundle obtained by restricting $q_\nu$. By pulling it back via $\phi_t$, we obtain a principal $T$-bundle $\phi_t^*L_t \longrightarrow N_{\mathsf{s}}$. We obtain a smooth family of principal $T$-bundles
$$
q \colon \mathcal{L}_\mathsf{s} \coloneqq \bigcup_{t \in [0,1]} \phi_t^* L_t \to N_\mathsf{s}.
$$
Thus the bundles $q_0$ and $q_1$ are homotopic.

Since the central fiber $\mathcal{X}_0$ is a toric variety, the bundle $q_0$ is trivial. Therefore $q_1$ is also trivial, and hence $L_\mathsf{s} \simeq T^{n+m}$. Finally, the monotonicity of $L_\mathsf{s}$ follows from Lemma~\ref{lemma_monotoneLagrangiantori}.
\end{proof}

\begin{remark}
In \cite{HL20}, Hoffman--Lane constructed completely integrable systems on $G/U$ by generalizing the construction of Harada--Kaveh in \cite{HK15}, which produces integrable systems on smooth projective varieties, to singular quasi-projective varieties. The Lagrangian tori $L_\mathsf{s}$ should coincide with fibers of these completely integrable systems.
\end{remark}

Conjecturally, the disk potential of the monotone Lagrangian torus $L_\mathsf{s}$ produces the mirror cluster chart 
$$
(\mathcal{X}_\mathsf{s}, W |_{\mathcal{X}_\mathsf{s}})
$$ 
where $W$ is the GHKK superpotential in~\eqref{equ_Wtheta}.

To describe the disk potential of $N_\mathsf{s}$, we briefly review the string polytope associated with a reduced expression $\underline{w}_0$. Consider $\R^m$ whose coordinates are ${u}_j$ for $j \in J$. The \emph{string polytope} $\Delta^\lambda_{\underline{w}_0}$ is defined by the intersection between the half spaces in $(1)$ together with the following inequalities~\eqref{equ_lambdacone}.

\begin{itemize}
\item[(3)] For each ascending array of chambers $\mathsf{r} \in \mathcal{R}_{\underline{w}_0}$, let $h_{\mathsf{r}}({\mathbf{u}})$ denote the sum of the chamber variables corresponding to the chambers in $\mathsf{r}$. If the array $\mathsf{r}$ lies in the $i$-th column, then define the half space by
\begin{equation}\label{equ_lambdacone}
(\lambda_i - \lambda_{i+1}) - h_{\mathsf{r}}({\mathbf{u}}) \geq 0.
\end{equation}
Let $\mathbf{v}_\mathsf{r}$ be the inward primitive normal vector to the half space~\eqref{equ_lambdacone}.
\end{itemize}

\begin{remark}
The string polytope $\Delta_{\underline{w}_0}^\lambda$ is obtained as the intersection of two rational polyhedral cones, the \emph{string cone} $\mathcal{C}_{\underline{w}_0}^s$  and the \emph{$\lambda$-cone} $\mathcal{C}_{\underline{w}_0}^{\lambda}$ introduced in~\cite{Lit98, BZ01}. Indeed, it is a Newton--Okounkov body given by a sequence of resolutions of Schubert varieties given by  ${\underline{w}_0}$ by Kaveh \cite{Kav15}. We follow the description of the string cone due to Gleizer--Postnikov \cite{GP00} and the description of the $\lambda$-cone due to Rusinko \cite{Rus08}, both of which are formulated in terms of wiring diagrams. Throughout this paper, we describe these cones in terms of chamber variables. We refer to \cite[Sections~2 and~4]{CKLP21} for further details.
\end{remark}

Suppose that the chamber variables in the $i$-th column, listed from top to bottom, are
$$
u_{\underline{i}},\, u_{i_1},\, \dots,\, u_{i_\kappa}.
$$
Define the vector $\mathbf{v}_i=(v_{i,j})_{j\in\underline{J}}$ by
$$
v_{i,j} =
\begin{cases}
1 &\mbox{if the chamber variable $u_j$ lies in the $i$-th column}, \\
0 &\mbox{otherwise.}
\end{cases}
$$

Suppose that $\mathsf{r} \in \mathcal{R}_{\underline{w}_0}(i)$ consists of chambers $u_{i_j},\,u_{i_{j+1}}, \dots, u_{i_\kappa}$. Let $\mathsf{q}$ be the descending array consisting of chambers $u_{\underline{i}}, u_{i_1}, \dots, u_{i_{j-1}}$. Then 
\begin{equation}\label{equ_inwardsum}
\mathbf{v}_\mathsf{r} + \mathbf{v}_i = \mathbf{v}_\mathsf{q}.
\end{equation}

Let $\mathbf{z}$ and $\mathbf{y}$ be the exponential variables corresponding to the chamber variables indexed by $J$ and $\underline{J}$, respectively. We are now ready to state the main theorem.

\begin{theorem}\label{theorem_GHKKdisk}
Suppose that the disk potential of $N_{\underline{w}_0}$ is given by
\begin{equation}\label{equ_diskNs}
W_{N_{\underline{w}_0}}(\mathbf{z}) =  \sum_{\mathsf{p} \in \mathcal{P}_{\underline{w}_0}} \mathbf{z}^{\mathbf{v}_\mathsf{p}} + \sum_{\mathsf{r} \in \mathcal{R}_{\underline{w}_0}}  \mathbf{z}^{\mathbf{v}_\mathsf{r}}.
\end{equation}
Then the disk potential of $L_{\underline{w}_0}$ is equal to the GHKK superpotential $W_{\underline{w}_0}$ in~\eqref{equ_GHKKssupter}.
\end{theorem}

\begin{remark}
When $\underline{w}_0$ is the standard reduced expression of $w_0$, the corresponding string polytope is unimodularly equivalent to the Gelfand--Zeitlin polytope by \cite{Lit98}. By \cite{NNU10}, the disk potential of $N_{\underline{w}_0}$ is computed and is of the form~\eqref{equ_diskNs} (up to a coordinate change). 

More generally, using the method of \cite{NNU10}, the disk potential function of $N_\mathsf{s}$ in $G/B$ is computed in \cite{CKLP23} and is precisely of the form~\eqref{equ_diskNs} in the case where a reduced expression of $w_0$ has small index. In particular, every reduced expression of $w_0$ has small index for $\mathrm{SL}_n(\C)/B$ with $n \le 5$, so the hypothesis of Theorem~\ref{theorem_GHKKdisk} is satisfied in these cases.
\end{remark}

\begin{proof}
Let $\underline{w}_0$ be a reduced expression of the longest element $w_0$ of $\mathfrak{S}_{n+1}$. 
We first recall the following results from \cite{CKLP21}. 
\begin{itemize}
\item \cite[Proposition 4.5]{CKLP21} (see also \cite{KS22})  For every rigorous path $\mathsf{p}$ and every ascending array $\mathsf{r}$, the inequalities in~\eqref{equ_stringincone} and~\eqref{equ_lambdacone} define facets of the string polytope. Namely, each linear function $h_{\mathsf{p}}$ and $h_{\mathsf{r}}$ supports the relative interior of a facet.
\item \cite[Proposition 5.6]{CKLP21} For each $j \in [m]$, there exists a rigorous path $\mathsf{p}_j$ such that 
\begin{equation}\label{equ_pjpeak}
h_{\mathsf{p}_j} (\mathbf{u}) = u_j + h^\prime_{\mathsf{p}_j}(u_{j+1}, \dots, u_m) \geq 0 \mbox{ for $j \in [m]$}
\end{equation}
where $h^\prime_{\mathsf{p}_j}$ is a linear function in the variables $u_{j+1}, \dots, u_m$.
\end{itemize} 

Let $\beta_j \in \pi_2(G/B, N_{\underline{w}_0})$ be the basic disk class corresponding to the facet defined by $h_{\mathsf{p}_j} = 0$ in~\eqref{equ_pjpeak}. We choose the classes $\beta_1, \dots, \beta_m$ so that their boundary classes form a basis of $\pi_1(N_{\underline{w}_0})$. We compute the boundary lifting map with respect to this basis.  

We divide the lifting process into two cases. Let $\beta \in \pi_2(G/B, N_{\underline{w}_0})$ with $n_\beta \neq 0$. Then $\beta$ is represented by a basic holomorphic disk corresponding either to a rigorous path $\mathsf{p} \in \mathcal{P}_{\underline{w}_0}$ or to an ascending array $\mathsf{r} \in \mathcal{R}_{\underline{w}_0}$.  

\begin{enumerate}
\item For each rigorous path $\mathsf{p} \in \mathcal{P}_{\underline{w}_0}$, let $\beta \in \pi_2(G/B,N_{\underline{w}_0})$ be the basic disk class corresponding to $\mathsf{p}$. We may write 
$$
\partial \beta = \sum_{j} b_j  \partial\beta_j.
$$
Then the associated spherical class is trivial, namely, 
$$
\alpha = \beta + \sum_{j} b_j \beta_j = 0 \in \pi_2(G/B).
$$
Thus, the Chern number of every line bundle on $\alpha$ is zero so that $\mathbf{a} = \mathbf{0}$. Therefore, by~\eqref{equ_linearcomofpartialbeta}, we have 
\begin{equation}\label{equ_111}
\partial \beta^\sharp = - \sum_{j} b_j \partial \beta_j^\sharp.
\end{equation}
\item Next, consider an ascending array $\mathsf{r} \in \mathcal{R}_{\underline{w}_0}$. Suppose that $\mathsf{r}$ lies in the $i$-th column. Let $\beta \in \pi_2(G/B, N_{\underline{w}_0})$ be the basic disk class corresponding to $\mathsf{r}$. As before, we may write 
$$
\partial \beta = \sum_{j} b_j \partial \beta_j.
$$
The associated capped spherical class $\alpha$ has symplectic area $\lambda_i - \lambda_{i+1}$. It follows that $\alpha = [X_{s_i}]$. Therefore, by~\eqref{equ_linearcomofpartialbeta}, we have
\begin{equation}\label{equ_222}
\partial \beta^\sharp = - \sigma_i - \sum_{j} b_j \partial \beta_j^\sharp.
\end{equation}
\end{enumerate}

Let $\mathbf{y}$ be the exponential variables corresponding to $\{ u_{\underline{i}} \mid i \in \underline{I} \} \cup \{ u_j \mid j \in J\}$. According to~\eqref{equ_diskNs}, ~\eqref{equ_111},~\eqref{equ_222}, and Theorem~\ref{theorem_maintheorem2}, the disk potential of $L_\mathsf{s}$ is
\begin{equation}\label{equ_lifteddiskpotential}
W_{L_\mathsf{s}}(\mathbf{y}) =  \sum_{\mathsf{p} \in \mathcal{P}_{\underline{w}_0}} \mathbf{y}^{-\mathbf{v}_\mathsf{p}} + \sum_{i=1}^n \sum_{\mathsf{r} \in \mathcal{R}_{\underline{w}_0}(i)}  \mathbf{y}^{-\mathbf{v}_\mathsf{r}} \cdot y^{-1}_{\underline{i}}.
\end{equation}
To identify it with GHKK superpotential~\eqref{equ_GHKKssupter}, we perform the coordinate change 
\begin{equation}\label{equ_coordinatechange}
y_{\underline{i}} \mapsto \mathbf{x}^{\mathbf{v}_i} \mbox{ for $i \in \underline{I}$}, \quad y_j \mapsto x_j \mbox{ for $j \in J$}.
\end{equation}
Using~\eqref{equ_inwardsum}, we recover the GHKK superpotential as follows.
\begin{align*}
W_{L_\mathsf{s}}((\mathbf{y}) |_{\eqref{equ_coordinatechange}} &=  \sum_{\mathsf{p} \in \mathcal{P}_{\underline{w}_0}} \mathbf{x}^{- \mathbf{v}_\mathsf{p}} + \sum_{i=1}^n \sum_{\mathsf{r} \in \mathcal{R}_{\underline{w}_0}(i)}  \mathbf{x}^{-\mathbf{v}_\mathsf{r}} \cdot {\mathbf{x}^{-\mathbf{v}_i}} = \sum_{\mathsf{p} \in \mathcal{P}_{\underline{w}_0}} \mathbf{x}^{- \mathbf{v}_\mathsf{p}} + \sum_{\mathsf{q} \in \mathcal{Q}_{\underline{w}_0}}  \mathbf{x}^{-\mathbf{v}_\mathsf{q}}.
\end{align*}
\end{proof}

We close this section with explicit examples.

\begin{example}
Consider $\underline{w}_0 = s_1 s_2 s_1 s_3 s_2 s_1$ in Example~\ref{example_twored} and see Figure~\ref{figure_wiringoriented} for the corresponding wiring diagram $G_{\underline{w}_0}$. By \cite[Corollary 6.14]{CKLP23}, the disk potential of $N_{\underline{w}_0}$ in $\mathrm{SL}_4(\C)/B$ is
{\small
\begin{align*}
W_{N_{\underline{w}_0}}(\mathbf{z}) = z_1z_2z_4 + z_2z_4 + z_4 + z_3 z_5 + z_5 + z_6 + \frac{1}{z_1z_3z_6} + \frac{1}{z_3z_6} + \frac{1}{z_6} + \frac{1}{z_2 z_5} + \frac{1}{z_5} + \frac{1}{z_4}.
\end{align*}
}
By~\eqref{equ_lifteddiskpotential}, the disk potential of $L_{\underline{w}_0}$ in $\mathrm{SL}_4(\C)/U$ is 
{\small
\begin{align*}
W_{L_{\underline{w}_0}}(\mathbf{y}) = \frac{1}{y_1y_2y_4} + \frac{1}{y_2y_4} + \frac{1}{y_4} + \frac{1}{y_3 y_5} + \frac{1}{y_5} + \frac{1}{y_6} + \frac{y_1y_3y_6}{y_{\underline{1}}} + \frac{y_3y_6}{y_{\underline{1}}} + \frac{y_6}{y_{\underline{1}}} + \frac{y_2 y_5}{y_{\underline{2}}} + \frac{y_5}{y_{\underline{2}}} + \frac{y_4}{y_{\underline{3}}}.
\end{align*}
}
The coordinate change~\eqref{equ_coordinatechange} is 
$$
y_{\underline{1}} \leftrightarrow x_{\underline{1}}x_1x_3x_6, \quad y_{\underline{2}} \leftrightarrow x_{\underline{2}}x_2x_5, \quad y_{\underline{3}} \leftrightarrow x_{\underline{3}}x_4, \quad y_j \leftrightarrow x_j \mbox{ for $j \in J$}.
$$
Under this coordinate change,  $W_{L_{\underline{w}_0}}(\mathbf{y})$ agrees with $W_{\underline{w}_0}(\mathbf{x})$ in Example~\ref{example_GHKKsuperpo}.

Consider $\underline{w}_{0^\prime} = s_2 s_1 s_3 s_2 s_1 s_3$ in Example~\ref{example_twored} and see Figure~\ref{figure_wiringoriented} for the corresponding wiring diagram $G_{\underline{w}_{0^\prime}}$. By \cite[Corollary 6.14]{CKLP23}, the disk potential of $N_{\underline{w}_{0^\prime}}$ in $\mathrm{SL}_4(\C)/B$ is
{\small
\begin{align*}
W_{N_{\underline{w}_{0^\prime}}}(\mathbf{z}) = z_1z_2z_3z_4 +z_2z_3z_4 + z_2z_4 + z_3z_4 + z_4 + z_5 + z_6 + \frac{1}{z_2 z_5} + \frac{1}{z_5} + \frac{1}{z_1z_4} + \frac{1}{z_4} + \frac{1}{z_3z_6} + \frac{1}{z_6} .
\end{align*}
}
By~\eqref{equ_lifteddiskpotential}, the disk potential of $L_{\underline{w}_{0^\prime}}$ in $\mathrm{SL}_4(\C)/U$ is 
{\small
\begin{align*}
W_{L_{\underline{w}_{0^\prime}}}(\mathbf{y}) = \frac{1}{y_1y_2y_3y_4} + \frac{1}{y_2y_3y_4} + \frac{1}{y_2y_4} + \frac{1}{y_3y_4} + \frac{1}{y_4} + \frac{1}{y_5} + \frac{1}{y_6} + \frac{y_2 y_5}{y_{\underline{1}}} + \frac{y_5}{y_{\underline{1}}} + \frac{y_1y_4}{y_{\underline{2}}} + \frac{y_4}{y_{\underline{2}}} + \frac{y_3y_6}{y_{\underline{3}}} + \frac{y_6}{y_{\underline{3}}}.
\end{align*}}
The coordinate change~\eqref{equ_coordinatechange} is 
$$
y_{\underline{1}} \leftrightarrow x_{\underline{1}}x_2x_5, \quad y_{\underline{2}} \leftrightarrow x_{\underline{2}}x_1x_4, \quad y_{\underline{3}} \leftrightarrow x_{\underline{3}}x_3x_6, \quad y_j \leftrightarrow x_j \mbox{ for $j \in J$}.
$$
Under this coordinate change,  $W_{L_{\underline{w}_{0^\prime}}}(\mathbf{y})$ agree with $W_{\underline{w}_{0^\prime}}(\mathbf{x})$ in Example~\ref{example_GHKKsuperpo}.
\end{example}

\providecommand{\bysame}{\leavevmode\hbox to3em{\hrulefill}\thinspace}
\providecommand{\MR}{\relax\ifhmode\unskip\space\fi MR }
\providecommand{\MRhref}[2]{%
  \href{http://www.ams.org/mathscinet-getitem?mr=#1}{#2}
}
\providecommand{\href}[2]{#2}

\end{document}